\documentclass[hidelinks,onefignum,onetabnum]{siamart250211}

\usepackage{lipsum}
\usepackage{amsfonts}
\usepackage{graphicx}
\usepackage{amsmath,amssymb}
\usepackage{xcolor}
\usepackage{verbatim}

\usepackage{algorithm}
\usepackage[noend]{algpseudocode}
\usepackage{bbm}
\usepackage[bb=boondox]{mathalfa}
\usepackage{graphicx} %
\usepackage{hyperref}
\usepackage{mathrsfs}

\newcommand{\A}{\mathbb{A}}
\newcommand{\AF}{\mathbf{A}}
\newcommand{\B}{\mathbb{B}}

\newcommand{\C}{\mathbb{C}}

\newcommand{\Sub}{\mathbf{S}}
\newcommand{\T}{\mathbf{T}}

\newcommand{\Z}{\mathbf{Z}}
\newcommand{\X}{\mathbf{X}}
\newcommand{\J}{\mathbb{J}}

\newcommand{\Y}{\mathbf{Y}}
\newcommand{\D}{\mathbf{D}}
\newcommand{\I}{\mathbb{I}}

\newcommand{\PL}{\text{\L}}
\newcommand{\bb}{\mathbb}
\newcommand{\scr}{\mathscr}

\newcommand{\RL}{\overline{R}}
\newcommand{\TS}{\mathcal{T}}
\newcommand{\NS}{\mathcal{N}}
\newcommand{\MN}{\mathcal{M}}

\newcommand{\prox}{\texttt{prox}}
\definecolor{xkchen}{rgb}{0.75, 0.2, 1.0}
\newtheorem{assumption}{Assumption}

\ifpdf
  \DeclareGraphicsExtensions{.eps,.pdf,.png,.jpg}
\else
  \DeclareGraphicsExtensions{.eps}
\fi

\newsiamremark{remark}{Remark}
\newsiamremark{hypothesis}{Hypothesis}
\crefname{hypothesis}{Hypothesis}{Hypotheses}
\newsiamthm{claim}{Claim}
\newsiamremark{fact}{Fact}
\crefname{fact}{Fact}{Facts}

\headers{Adaptive Decentralized Composite Optimization}{X. Chen, I. Kuruzov, and G. Scutari}

\title{Adaptive   Decentralized Composite Optimization via Three-Operator Splitting
\thanks{Submitted to the editors February, 2026.
\funding{The work of Chen and Scutari    has been supported by   the Office of Naval Research (ONR Grant N. N000142412751).}}}

\author{Xiaokai Chen\thanks{Edwardson School of Industrial Engineering, Purdue University, West Lafayette, IN 47906 USA
(\email{chen4373@purdue.edu} and \email{gscutari@purdue.edu}, 
  ).}
\and Ilya Kuruzov 
\and Gesualdo Scutari\footnotemark[3]}

\usepackage{amsopn}

\ifpdf
\hypersetup{
  pdftitle={An Example Article},
  pdfauthor={D. Doe, P. T. Frank, and J. E. Smith}
}
\fi

\begin{document}
\hypersetup{colorlinks=true,
            linkcolor=black,
            anchorcolor=black,
            citecolor=black}
\maketitle

\begin{abstract}
The paper studies decentralized optimization over networks, where agents minimize a sum of {\it locally} smooth (strongly) convex   losses and   plus a nonsmooth convex extended value term. We propose  decentralized methods   wherein  agents {\it adaptively} adjust their stepsize via local backtracking procedures coupled with lightweight   min-consensus protocols. Our design stems from a three-operator splitting factorization applied to an equivalent reformulation of the problem. The reformulation is endowed with a new BCV preconditioning  metric (Bertsekas–O'Connor–Vandenberghe),   which enables efficient decentralized implementation and local stepsize adjustments. We establish robust convergence guarantees. Under mere convexity, the proposed methods converge with a sublinear rate.  Under strong convexity of the sum-function, and assuming the nonsmooth component is partly smooth, we further prove linear convergence.   Numerical experiments corroborate the theory and highlight the effectiveness of the proposed adaptive stepsize strategy.
\end{abstract}

\begin{keywords}
Adaptive stepsize, convex optimization, decentralized optimization, networks.   
\end{keywords}


 \section{Introduction}
\label{sec:intro}
We study decentralized optimization problems in the form \vspace{-0.1cm}
\begin{equation}
    \label{eq:problem}
    \tag{P}
    \min_{x\in \bb{R}^d}\quad   u(x):=\underbrace{\sum_{i=1}^mf_i(x)}_{:=f(x)}+\underbrace{\sum_{i=1}^m r_i(x)}_{:=r(x)},\vspace{-0.1cm}
\end{equation}
where   $f_{i} : \mathbb{R}^{d} \to \mathbb{R}$ \ is the  loss  of agent $i\in [m]:=\{1,\ldots, m\}$, assumed to be (strongly) convex and \emph{locally} smooth, and \ $r_i : \mathbb{R}^{d} \to \mathbb{R}\cup \{-\infty,+\infty\}$ \ is a convex, nonsmooth proper extended value function. Both $f_i$ and $r_i$ are private functions assumed to be known only to agent $i$. Agents are embedded in a  communication network, modeled as a fixed, undirected and connected graph $\mathcal{G}$, with no central servers.

 Problem~\eqref{eq:problem} arises in  several applications of interest, including signal processing, machine learning, multi-agent systems, and communications. The literature abounds with decentralized  methods for~\eqref{eq:problem} under the standing assumption that the $f_i$’s are {\it globally} smooth and {\it identical} $r_i=r$ for all $i$; we refer to the  tutorials (and reference therein) \cite{Nedic_Olshevsky_Rabbat2018,Xin_Pu_Nedic_Khan_2020} and monograph \cite{Sayed-book} 
 for comprehensive reviews. Global smoothness aside, these methods  impose conservative stepsize bounds that depend on parameters such as the global Lipschitz constants of the agents' gradients, the spectral gap of the graph gossip matrix,  and other topological network properties. Such information is generally unavailable locally to agents in real-world deployments. As a result, stepsizes are frequently selected via manual tuning, yielding performance that is unpredictable, problem-dependent, and difficult to reproduce. Moreover, these approaches may fail when the agents' losses are only \emph{locally} smooth; see Sec.~\ref{subsec:related works} for further discussion. This paper addresses these limitations by proposing adaptive-stepsize decentralized algorithms that solve~\eqref{eq:problem} under local smoothness, {\it private}, nononsmooth functions $r_i$, and without requiring global problem/network information. 

\subsection{Related works}
\label{subsec:related works}
\textit{1. Adaptive centralized methods:} There has been a growing interest in developing  adaptive stepsize    methods   in  centralized  optimization. Representative  examples include   line-search    \cite{bertsekas2016nonlinear}, Polyak's   \cite{polyak1969minimization} and  Barzilai-Borwein's   \cite{BarzilaiBorwein1988} rules, curvature based estimate rules 
\cite{Malitsky2019AdaptiveGD,Malitsky_Mishchenko_2024,Zhou24}, and  adaptive gradient methods such as AdaGrad/Adam/AMSGrad and variants \cite{duchi2011adaptive,kingma2014adam,Reddi2018OnTC,Orabona19,Ward20,jordan6muon}. 
However, these schemes are generally ill-suited for \emph{mesh networks} since they rely on a central server to aggregate updates and enforce a common stepsize. 
Among the above, only~\cite{bertsekas2016nonlinear,polyak1969minimization,duchi2011adaptive} naturally handle nonsmooth, identical  convex terms   $r$. 

{\it 2. Adaptive decentralized   methods:} The landscape of adaptive  {\it decentralized} methods is comparatively limited  
\cite{Nazari_Tarzanagh_Michailidis_2022,chen2023convergence, li2024problem,kuruzov2024achieving,Zhou24,Notarstefano24}. Most existing works focus on stochastic/online (non)convex {\it smooth} optimization (i.e., $r_i\equiv 0$), and obtain adaptivity through gradient normalization or history-based scaling rules
\cite{Nazari_Tarzanagh_Michailidis_2022,chen2023convergence,li2024problem}.  
Except for \cite{li2024problem}, these approaches assume globally Lipschitz continuous gradients, which supports convergence under standard diminishing stepsizes of order $\mathcal{O}(1/\sqrt{k})$ (with $k$ denoting the iteration index). 
Moreover, \cite{Nazari_Tarzanagh_Michailidis_2022,chen2023convergence} require {\it a priori} knowledge of problem-dependent constants. 
The recent work~\cite{Notarstefano24} introduces a Port--Hamiltonian framework for smooth, strongly convex, unconstrained problems: while the method is parameter-free in the centralized setting, decentralized convergence (global asymptotic stability) is shown only under specific graph topologies (e.g., complete or ring graphs) or graph-dependent stepsize restrictions requiring global network information; additionally, no explicit convergence rate is provided. 
   Finally, \cite{kuruzov2024achieving} proposes an adaptive decentralized method for smooth, strongly convex instances of~\eqref{eq:problem}, proving linear convergence and improved theoretical and practical performance relative to non-adaptive baselines.

\paragraph{Composite objectives and local smoothness.} None of the above decentralized methods addresses {\it composite} problems of the form~\eqref{eq:problem}, nor do they accommodate instances where  $f_i$'s are only {\it locally} smooth.  While adaptive proximal schemes for composite optimization exist in centralized settings \cite{Yura-Pock-LSPrimal-dual-18,Latafat_23b}, their decentralization  faces substantial limitations.
   In particular,~\cite{Yura-Pock-LSPrimal-dual-18} would require, at {\it every} line-search trial:
{\bf (i)} transmitting ambient-dimensional vectors, leading to prohibitive communication overhead; and
{\bf (ii)} computing (and disseminating) the smallest trial stepsize over the {\it entire} network.
  The method in \cite{Latafat_23b}  avoids line-search, but still requires computing and propagating {\it global} scalar quantities at each iteration, and its stepsize selection depends on a {\it global} network constant that is not locally available. 
Conservative local upper bounds of this parameter  destroy adaptivity, slow convergence, and make performance sensitive to tuning parameters--this is  corroborated by the numerical results in Sec.~\ref{sec:numeric}. 

 Overall, existing adaptive decentralized methods do not solve composite optimization problems as~\eqref{eq:problem} using \emph{only neighbor-to-neighbor communications}, let alone under merely local smoothness. To the best of our knowledge, the algorithms proposed in this paper are the first parameter-free decentralized methods that close this gap. 
 \vspace{-.1cm}

\subsection{Major contributions}
 Our main contributions are summarized as follows. 
\label{subsec:contribution}
\textit{1. Algorithm design:} We introduce an adaptive stepsize  decentralized algorithm tailored to the composite structure of~\eqref{eq:problem}. To handle the nonsmooth terms, our design   builds on a Davis-Yin three-operator splitting applied to a BCV-type reformulation of~\eqref{eq:problem}~\cite{bertsekas2016nonlinear,o2020equivalence}. 
A key novelty is the design of an appropriate \emph{metric} in the BCV transformation that enables a fully decentralized implementation supporting \emph{local backtracking} stepsize updates at each agent while using only \emph{single-hop} communications. 
    Endowing decentralized methods with backtracking is  challenging:   
there is no canonical notion of a descent direction available to each agent;  agents' updates are coupled through the network and are not, by themselves, optimization steps for the global objective $u$. We resolve this issue by identifying a suitable agent-wise update direction and a backtracking test on local agents' losses that, via a descent inequality of a properly chosen Lyapunov function, certifies global convergence. 
To coordinate   stepsizes, we propose two implementations based on \emph{global} and \emph{local} min-consensus protocols, respectively; they complement applications in    different networking settings. Remarkably, neither implementation requires knowledge of global optimization constants or global network parameters.

\textit{2. Convergence guarantees:} We provide a comprehensive convergence analysis for both proposed algorithms.  We prove that all agents' iterates consensually converge to a solution of~\eqref{eq:problem}. Under mere convexity of the $f_i$'s, we establish a sublinear rate of order $\mathcal{O}(1/k)$ for a suitable optimality gap.

\textit{3. Linear convergence rate with strongly convex $f_i$ and partly smooth $r_i$:} When each $f_i$ is {\it locally} strongly convex and each $r_i$ is partly smooth relative to a $\mathcal{C}^2$ embedded manifold (Def.~\ref{def:partial smoothness}) around the limit point of the algorithms, we prove linear convergence. 
 Our results are twofold:  \textbf{(i)}  if the aforementioned active manifold   is \emph{affine}, we establish finite-time manifold identification followed by a global linear rate, together with an explicit iteration-complexity bound; 
{\bf (ii)}   if the active manifold is a general $\mathcal{C}^2$ manifold, we obtain \emph{asymptotic} linear convergence.

\textit{4. An adaptive three-operator splitting:} As a by-product of our design, we also obtain an adaptive three-operator splitting algorithm (a backtracking variant of Davis--Yin splitting) for composite optimization in the form \eqref{eq:problem}, with locally smooth  $f$, implementable in centralized or federated (master/client)    architectures. 
This  new method inherits the same convergence guarantees as its decentralized counterpart discussed above (global convergence under convexity and linear convergence under strong convexity plus partial smoothness), while remaining parameter-free (no knowledge of any  optimization parameter is required). 
We believe this scheme is of independent interest as an adaptive splitting primitive for large-scale composite optimization.

In addition to the above guarantees, numerical experiments show that the proposed adaptive methods significantly outperform existing decentralized algorithms applicable to~\eqref{eq:problem}, which rely on \emph{non-adaptive} (conservative) stepsize choices. 

 A preliminary version of this work appeared in~\cite{chen2025parameter}.
The present paper substantially extends~\cite{chen2025parameter}  by providing: {\bf (i)}  
  more principled and less restrictive adaptive stepsize rules based on a new merit function and an inexact descent analysis, leading to faster algorithms;  {\bf (ii)} a more comprehensive convergence theory, including iterate convergence and linear rates under local strong convexity of the $f_i$'s and   partial smoothness of the $r_i$'s; and
   {\bf (iii)}  complete proofs and   expanded experiments.
  
A related preprint appeared during the preparation of this manuscript~\cite{xu2025accelerated}, and after \cite{chen2025parameter}.
It develops an adaptive decentralized algorithm for a variant  of~\eqref{eq:problem} including   conic constraints; an accelerated sublinear convergence is shown for convex losses and convex constraints.  Adaptivity is achieved via local backtracking procedures coupled with a \emph{global max-consensus} mechanism.
Our method differs from ~\cite{xu2025accelerated} in the following key features: 
{\bf (i) } it can handle {\it locally smooth} losses $f_i$'s  (rather than globally smooth); 
{\bf (ii)} adaptive stepsize updates can be implemented using only neighboring communications; and    %
{\bf (iii)} it provably achieves   linear convergence (in the sense discussed above) when the agents' losses are strongly convex.

\subsection{Notation and paper organization}
\label{subsec:notation} Let $\overline{\mathbb{R}}:=(-\infty,+\infty]$; $[m]:=\{1,\ldots,$ $m\}$;     ${1_m}$ and ${0_m}$  denote the all-one and all-zero $m$-dimensional vectors (dimension omitted when clear);   $[x]_+:= \max(x, 0)$, intended component-wise for $x \in \bb{R}^d$; the fraction $0/0$ is interpreted as $+\infty$.
Capital letters denote matrices  while boldface capital letters denote stacked agent variables, e.g.,
${\bf X}=[x_1,\ldots,x_m]^\top$, where the $i$-th row corresponds to agent $i$. We use calligraphic/script letters to denote    stacked multiple matrices. 
Let $\bb{S}^m$, $\bb{S}^m_{+}$ , and $\bb{S}^m_{++}$ be the set of $m\times m$ (real) symmetric, symmetric positive semidefinite, and symmetric positive definite matrices, respectively;
  ${\tt null}(A)$ and ${\tt span}(A)$ denote the nullspace and range of $A$, respectively. The eigenvalues of $A \in \bb{S}^m$ are ordered in nonincreasing order,   denoted by $\lambda_{\max}(A):=\lambda_1(A)\geq\cdots\lambda_m(A)=:\lambda_{\min}(A)$. Let $\langle X, Y \rangle := {\tt tr}(X^{\top}Y )$, for any $X$ and $Y$ of suitable size, where $\tt{tr}(\bullet)$ is the trace operator;  $\|X\|_M := \sqrt{\langle MX,X\rangle}$, for any $M\in\bb{S}^m_{++}$ and $X\in\bb{R}^{m\times d}$.  For an    operator $\mathbb{T}$,   ${\tt Fix}(\bb{T})$ is the set of its fixed points; and  $\J_{\mathbb{T}}:=(\I+\mathbb{T})^{-1}$  
 is the resolvent of  $\mathbb{T}$.   For two operators $\bb{A}$ and $\bb{B}$, $(\bb{A}\circ\bb{B})(\bullet)$ stands for $\bb{A}(\bb{B}(\bullet))$; $\bb{I}$ is the identity operator and $\bb{0}$ represents the zero operator.

For a nonempty set ${\cal C}$, $\texttt{aff}({\cal C})$, and $\texttt{conv}({\cal C})$,   denote its affine hull, convex hull, respectively; $\texttt{par}({\cal C})=\texttt{span}(\mathcal C-\mathcal C)$ is the subspace parallel to a convex set ${\cal C}$.
  For an extended-value function ${\cal L}$, $\texttt{dom}({\cal L})$ denotes its effective domain. We use $\delta_{\cal{C}}$ to denote the indicator function on the set $\cal{C}$ and $f^*$ stands for the Fenchel conjugate of function $f$, namely, $f^*(y):=\sup_{x}\left\{\langle y,x\rangle-f(x)\right\}$. Given $x\in\bb{R}^d$, $\alpha>0$ and $r:\bb{R}^d\rightarrow\bb{R}\cup\{-\infty,\infty\}$, $\tt{prox}_{\alpha r}(x)$ denotes the proximal operator, defined as  ${\tt prox}_{\alpha r}(x):={\tt argmin}_{y\in \bb{R}^d} r(y)+\frac{1}{2\alpha}\|x-y\|^2.$  Given a subspace $\mathcal{T}$ of $\mathbb{R}^d$, $P_{\mathcal{T}}(x)$ denotes the unique orthogonal projection of $x\in\mathbb{R}^d$ onto $\mathcal{T}$.  For some matrix-valued functions $g,h$, we write $g(X)=o(\|h(X)\|)$ if and only if $\lim_{h(X)\rightarrow {0}}  {\|g(X)\|}/{\|h(X)\|}=0$. 
Let $\MN$ be a $\mathcal{C}^2$-smooth manifold of $\mathbb{R}^d$ around a point $x$. For any $x'\in \MN$ near $x$, we denote by $\TS_{\MN}(x')$  and  $\NS_{\MN}(x')$  the tangent  and  normal spaces of $\MN$ at  $x'$, respectively. 

Sec.~\ref{sec:algorithm design} presents  the design of the proposed algorithm along with its decentralized implementation based on global-min consensus protocols on the   agents'   stepsize values  (Algorithm~\ref{alg:DATOS}). Sec.~\ref{sec:convergence cvx} analyzes   Algorithm~\ref{alg:DATOS} under convexity and establishes a sublinear convergence rate. Sec.~\ref{sec:local min} introduces    a variant of Algorithm~\ref{alg:DATOS} (Algorithm~\ref{alg:DATOS_local}) implementable using only neighboring communications running   min-consensus, and proves its convergence.     Sec.~\ref{sec:linear}   establishes linear convergence  of the proposed algorithms when the $f_i$’s are strongly convex and $r$ is partly smooth. Numerical experiments are reported in Sec.~\ref{sec:numeric}.\vspace{-0.1cm}

\section{Algorithm Design}
\label{sec:algorithm design}
We study  Problem~\eqref{eq:problem} 
under the following  assumptions.

\begin{assumption} 
\label{ass:function}
The objective function in~\eqref{eq:problem} satisfies the following:

\noindent(\textbf{i}) Each $f_i:\mathbb{R}^d\rightarrow \mathbb{R}$ is convex on $\mathbb{R}^d$  and   {{\it locally}} smooth; 

\noindent(\textbf{ii}) Each $r_i:\mathbb{R}^d\rightarrow\mathbb{R}\cup\{-\infty,\infty\}$ is convex, proper and lower-semicontinuous; and 

\noindent(\textbf{iii}) $f + r$ is lower bounded. 

\end{assumption}

\begin{assumption}
\label{ass:graph}The network underlying Problem~\eqref{eq:problem} is modeled as an undirected, connected graph $\mathcal G=(\mathcal{V},\mathcal{E})$ with diameter $d_{\mathcal G}\geq 0$, where $\mathcal{V}=[m]$ and $(i,j)\in\mathcal{E}$ if and only if there is an edge between agent $i$ and $j$.  
\end{assumption}

\begin{definition}
    [gossip  matrices]
    \label{ass:W}
     $W_{\mathcal{G}}$ denotes the set of symmetric, doubly stochastic, gossip matrices $\widetilde{W}:=(\widetilde{w}_{ij})_{i,j=1}^m$ that are compliant with   $\mathcal{G}$, i.e.,   $\widetilde{w}_{ii}>0$, for all $i\in[m]$;   $\widetilde{w}_{ij}>0$ for all $(i,j)\in\mathcal{E}$; and $\widetilde{w}_{ij}=0$ otherwise. 
\end{definition}
Gossip matrices  are standard in the   decentralized  optimization literature, e.g., \cite{Nedic_Olshevsky_Rabbat2018,Sayed-book}.

\subsection{A three-operator splitting-based algorithm}
\label{subsec:TOS}
The proposed   algorithm design hinges on the following reformulation of \eqref{eq:problem} {\it leveraging the BCV technique}: introducing local copies $x_i\in \mathbb{R}^d$ of  $x$ and slack variables $\tilde x_i\in \mathbb{R}^d$,
the   program \vspace{-0.1cm}
\begin{equation}
    \label{eq:ausiliary}
    \tag{P'}
    \min_{\X,\widetilde{\X}\in\mathbb{R}^{m\times d}} \underbrace{F(\X)}_{:= \widetilde{F}(\X,\widetilde{\X})}+\underbrace{R(\X)+\delta_{\{0\}}(\widetilde{\X})}_{:=\widetilde{R}(\X,\widetilde{\X})}+\underbrace{\delta_{\{0\}}(\PL \X+M\widetilde{\X})}_{:=\widetilde{G}(\X,\widetilde{\X})},
\end{equation}
is equivalent to the  original Problem~\eqref{eq:problem}. In \eqref{eq:ausiliary},   we defined   
 $$ F(\X):=\sum_{i=1}^m f_i(x_i),\quad  R(\X):=\sum_{i=1}^m r_i(x_i);$$  $\delta_{\{0\}}:\mathbb{R}^{m\times d}\to \mathbb{R}\cup \{\infty\}$ is the indicator function of  $\{0\}$;  $\PL\in\mathbb{S}^{m}$   satisfies  $\texttt{null}(\PL)=\texttt{span}(1_m)$,  and $M\in\mathbb{S}^{m}_{++}$ is the BCV metric.  The matrix  $\PL$    enforces consensus among the agents' variables $x_i$'s via  $\PL \mathbf X=0$.  The choice of $M$    will be shown to be  crucial to enable   a fully decentralized implementation of the proposed algorithm as well as  adaptive local stepsize selection.  This justifies the presence of $\widetilde{G}$ in (\ref{eq:ausiliary}), 
 which is unconvential  in the classical decentralized optimization literature.

Under Assumption~\ref{ass:function}, finding a solution of Problem~(\ref{eq:ausiliary}) is equivalent to solving:\vspace{-0.1cm}
\begin{equation}
    \label{eq:inclusion 2}
  \text{find } \scr{X}:=\begin{bmatrix}
        \X\\\widetilde{\X}
    \end{bmatrix}\in \mathbb{R}^{2m\times d} \quad \text{ such that }\quad   0\in({\A}+{\B}+{\C})  \scr{X},
\end{equation}
where
$ {\A}:=\partial\widetilde{G}$,     ${\B}:=\partial \widetilde{R}$ and  $ {\C}:=\nabla \widetilde F$. 

 Invoking the  Davis-Yin three-operator splitting~\cite{davis2017three},  (\ref{eq:inclusion 2}) is equivalent to  \vspace{-0.1cm}
 \begin{equation}  \label{eq:inclusion 3}\text{find } \scr{X}, \scr{U}\in \mathbb{R}^{2m\times d}  \,\text{ s.t. } \,\mathbb{D}_{\alpha} \scr U=\scr U \,\text{ and } \,\scr{X}=   \mathbb{J}_{\alpha\B} \scr U,\end{equation}
where\vspace{-0.3cm}\begin{equation}\label{eq:D_op}\mathbb{D}_{\alpha}:=\I-\mathbb{J}_{\alpha \B}+\mathbb{J}_{\alpha \A}\circ(2\J_{\alpha\B}-\I-\alpha\C\circ\mathbb{J}_{\alpha\B}),\end{equation}    and   $\alpha>0$  plays the role of the stepsize.  Notice that  $\J_{\alpha \partial \widetilde R}$ (resp. $\J_{\alpha \partial \widetilde G}$)   is equivalent to the  proximal operator $\texttt{prox}_{\alpha\widetilde{R}}$ (resp. $\texttt{prox}_{\alpha\widetilde{G}}$). 
 
 One can solve    (\ref{eq:inclusion 3}) via the  fixed-point  Krasnosel'skii-Mann iteration~\cite{davis2017three}. Specifically, denoting by ${\scr A}^{k}:=[\AF^{k};\widetilde{\AF}^{k}]\in\bb{R}^{2m\times d}$ the intermediate variable obtained by applying   $\bb J_{\alpha\bb A}$, and allowing  iteration-dependent    stepsize values $\alpha^k>0$, we have: 
 \vspace{-0.2cm}
\begin{equation}
\label{eq:FPI_1}
      \begin{aligned}
    &{\scr X}^{k}= \texttt{prox}_{\alpha\widetilde{R}}\scr U^k,\\
&{\scr A}^{k+1}=\texttt{prox}_{\alpha\widetilde{G}}\left(2{\scr X}^{k}-{\scr U}^k -\alpha \nabla \widetilde F ({\scr X}^{k})\right),\\
& {\scr U}^{k+1}={\scr U}^{k}+  {\scr A}^{k+1}-{\scr X}^{k}.
\end{aligned} 
 \end{equation}

Next, we    design adaptive rules for  $\{\alpha^k\}_{k\ge 0}$ that guarantee  convergence of \eqref{eq:FPI_1}.

\subsection{Adaptive stepsize selection}
\label{subsec:stepsize selection}
The proposed approach consists   to identifying a suitable merit function that assesses the convergence of \eqref{eq:FPI_1}, and then devising stepsize update rules that guarantee a (possibly inexact) descent property for such a measure. To this end, we introduce the Lagrangian function  
\begin{equation}\vspace{-0.1cm}
    \label{eq:new Lagarangian}
    \mathcal{L}({\scr X},{\scr S}):=\widetilde{F}({\scr X})+\widetilde{G}({\scr X})+\langle{\scr S},{\scr X}\rangle-\widetilde{R}^*({\scr S}),
\end{equation}
 and interpret~\eqref{eq:FPI_1} as a primal--dual scheme for the saddle-point reformulation of~\eqref{eq:ausiliary}:   $$\min_{{\scr X}\in\mathbb{R}^{2m\times d}}\max_{{\scr S}\in\mathbb{R}^{2m\times d}}\mathcal{L}(\scr{X},\scr{S}),$$ where    ${\scr X}$ are the primal variables  and   ${\scr S}$  are the dual variables, the latter defined absorbing the ${\scr U}$-variable in~\eqref{eq:FPI_1} as 
 $${\scr S}=\begin{bmatrix}
      \Sub\\
      \widetilde{\Sub}
  \end{bmatrix}:=\frac{1}{\alpha}\left({\scr U}-{\scr X}\right).$$  
  Using the ${\scr S}$-variables,~\eqref{eq:FPI_1} reads $({\scr X}^{k+1},{\scr S}^{k+1}):={\bb T}_{\alpha^k}({\scr X}^{k},{\scr S}^{k})$, where\vspace{-0.1cm}
\begin{subequations}
\begin{equation}\label{eq:ATOS}
        {\scr X}^{k+1}=\prox_{\alpha^k\widetilde{R}}\left({\scr A}^{k+1}+\alpha^k{\scr S}^k\right), \quad 
{\scr S}^{k+1} ={\scr S}^k+\frac{1}{\alpha^k}({\scr A}^{k+1}-{\scr X}^{k+1}),\end{equation}
     \vspace{-0.2cm}
    \begin{equation}
         \label{eq:ATOS A}
{\scr A}^{k+1}=\prox_{\alpha^k\widetilde{G}}\left({\scr X}^k-\alpha^k{\scr S}^k-\alpha^k\nabla\widetilde{F}({\scr X}^k)\right).
    \end{equation}
\end{subequations}

  Denoting by ${\cal P}^*\times{\cal D}^*$  the set of saddles points of (\ref{eq:new Lagarangian}), we propose the following merit function: 
  given $({\scr X}^*,{\scr S}^*)\in {\cal P}^*\times{\cal D}^*$ and $\{\alpha^k\}_{k\geq 0}$,  let
\begin{equation} \label{eq:Lyapunov_0}{\cal V}_{{\scr X}^*{\scr S}^*}^{k}:=\|{\scr X}^{k}-{\scr X}^*\|^2+(\alpha^{k-1})^2\|{\scr S}^{k}-{\scr S}^*\|^2,\quad   k\geq 0.\end{equation}
Convergence of ${\cal V}_{{\scr X}^*{\scr S}^*}^k$ is established in Lemma~\ref{lemma:lyapunov asymptotic}, which builds on Lemma~\ref{lemma:distance characterization},  \ref{lemma:descent lemma}. 
\begin{lemma}\label{lemma:distance characterization}
  Suppose   Assumption~\ref{ass:function} holds. For any given  $({\scr A}^{k},{\scr X}^{k},{\scr S}^{k})$,  $\alpha^k>0$, and  fixed $({\scr X},{\scr S})\in  \texttt{dom}({\cal L})$, $({\scr A}^{k+1},{\scr X}^{k+1},{\scr S}^{k+1})$ given by \eqref{eq:ATOS} satisfies:
\begin{subequations}\label{eq:distance characterization}
       \begin{align}
       \label{eq:primal distance}
       \|{\scr X}^{k+1}-{\scr X}\|^2\leq& \|{\scr X}^k-{\scr X}\|^2+\|{\scr A}^{k+1}-{\scr X}^{k+1}\|^2\\
       \tag*{}
&-2\alpha^k\left(\mathcal{L}({\scr A}^{k+1},{\scr S}^{k+1})-\mathcal{L}({\scr X},{\scr S}^{k+1})\right)\\
\tag*{}
&-\|\widetilde{\AF}^{k+1}-\widetilde{\X}^k\|^2-\left(1-\alpha^k L^k\right)\|\AF^{k+1}-\X^k\|^2,\\
       \label{eq:dual distance}
           \|{\scr S}^{k+1}-{\scr S}\|^2\leq& \|{\scr S}^k-{\scr S}\|^2-\left(\frac{1}{\alpha^k}\right)^2\|{\scr A}^{k+1}-{\scr X}^{k+1}\|^2\\
           \tag*{}
           &+\frac{2}{\alpha^k}\left(\mathcal{L}({\scr A}^{k+1},{\scr S}^{k+1})-\mathcal{L}({\scr A}^{k+1},{\scr S})\right),
       \end{align}
   \end{subequations}
   where $L^k$ is the local estimate of the curvature of $F$ at $\X^k$ along    $\AF^{k+1}-\X^k$:
 \begin{equation}
     \label{eq:Lk}
     L^k:=2\,\frac{  F(\AF^{k+1})-F(\X^k)-\langle\nabla F(\X^k),\AF^{k+1}-\X^k\rangle }{\|\AF^{k+1}-\X^k\|^2}\geq 0.
 \end{equation}
\end{lemma}
\begin{proof}
  See Sec.~\ref{subsec:proof of lemma 1}.
\end{proof}
Invoking   
\[
{\cal L}({\scr X},{\scr S}^*)-{\cal L}({\scr X}^*,{\scr S})\ge 0,\quad 
\forall({\scr X},{\scr S})\in\texttt{dom}({\cal L}),
\] 
one infers from Lemma~\ref{lemma:distance characterization}, the following inexact descent   of   
${\cal V}_{{\scr X}^*{\scr S}^*}^{k}$.   
\begin{lemma}
    \label{lemma:descent lemma}
   Under the conditions of Lemma~\ref{lemma:distance characterization}, if $\alpha^k$ satisfies~\eqref{eq:Lk}, then for all $k\ge 0$ and any given $({\scr X}^*,{\scr S}^*)\in{\cal P}^*\times{\cal D}^*$,  
    \begin{equation}
        \label{eq:descent inequality}
\begin{aligned}
            {\cal V}_{{\scr X}^* {\scr S}^*}^{k+1}\leq{\cal V}_{{\scr X}^* {\scr S}^*}^k-&\|\widetilde{\AF}^{k+1}-\widetilde{\X}^k\|^2-(1-\alpha^kL^k)\|\AF^{k+1}-\X^k\|^2-2\alpha^k{\cal G}_{{\scr X}^*{\scr S}^*}^{k+1}\\
            &+\left[(\alpha^k)^2-(\alpha^{k-1})^2\right]\|{\scr S}^k-{\scr S}^*\|^2,
\end{aligned}
    \end{equation}
    where ${\cal G}_{{\scr X}^*{\scr S}^*}^{k+1}:={\cal L}({\scr A}^{k+1},{\scr S}^*)-{\cal L}({\scr X}^*,{\scr S}^{k+1})\geq 0$.
\end{lemma} 
In the {\it inexact} descent of ${\cal V}_{{\scr X}^* {\scr S}^*}^{k}$  above, the only potentially positive contribution is
\(
\bigl[(\alpha^k)^2-(\alpha^{k-1})^2\bigr]\|{\scr S}^k-{\scr S}^*\|^2.
\) To control this term while   keeping the negative residuals in~\eqref{eq:descent inequality} effective, we propose the following updating rule for stepsize $
\alpha^k$:    
\begin{equation}\label{eq:adaptive stepsize selection rule} 
    0<\alpha^k\leq \frac{\delta}{L^k},\quad (\alpha^k)^2-(\alpha^{k-1})^2\leq  
      \min\left\{\frac{1-\delta}{4}\frac{\|\AF^{k}-\X^{k-1}\|^2}{\|{\scr S}^k-{\scr S}^0\|^2}, {n^k}\right\},
\end{equation}
where $\{n^k\geq 0 \}_{k\geq -1}$ is a sequence to be properly chosen,  and  $\delta\in(0,1)$.   
The first condition,  \(1-\alpha^kL^k\ge 1-\delta>0\),   ensures a uniform decrease contributed of $   {\cal V}_{{\scr X}^* {\scr S}^*}^{k+1}$ through the residual term \(\|\AF^{k+1}-\X^k\|^2\) while the second condition   
controls the  term
$[(\alpha^k)^2-(\alpha^{k-1})^2]\|{\scr S}^k-{\scr S}^*\|^2$ on the RHS of~\eqref{eq:descent inequality}, by
\begin{equation}\label{eq: upper-bound-positive-term}
\bigl[(\alpha^k)^2-(\alpha^{k-1})^2\bigr]\|{\scr S}^k-{\scr S}^*\|^2
\le \frac{1-\delta}{2}\|\AF^{k}-\X^{k-1}\|^2+ 2\, n^k \|{\scr S}^0-{\scr S}^*\|^2. 
\end{equation}
The first term on the RHS above can be absorbed    by a one-step shift of the Lyapunov function: ${\cal V}_{{\scr X}^* {\scr S}^*}^{k}$ is replaced by  the companion     $ {\cal V}_{{\scr X}^* {\scr S}^*}^{k}+ \frac{1-\delta}{2}\|\AF^{k}-\X^{k-1}\|^2$.   
 The second term in (\ref{eq: upper-bound-positive-term}) is handled by requesting a summable  budget   \(\sum_k  n^k<\infty\), given the lack of knowledge of   ${\scr S}^*$. This motivates the introduction of the  merit function: \vspace{-0.1cm} 
\begin{equation}\label{eq:companion-Lyapunov}
    \widetilde{{\cal V}}_{{\scr X}^\star{\scr S}^\star}^k:={\cal V}^{k}_{{\scr X}^\star{\scr S}^\star}+\frac{1-\delta}{2}\|\AF^k-\X^{k-1}\|^2- 2\|{\scr S}^0-{\scr S}^\star\|^2\sum_{t=-1}^{k-1} n^t,\vspace{-0.2cm} 
\end{equation} 
which convergences monotonically along the iterates \eqref{eq:ATOS}, as summarized next.  \begin{lemma}   
\label{lemma:lyapunov asymptotic}Let  $\{({\scr A}^{k},{\scr X}^{k},{\scr S}^{k})\}$ be the sequence generated  by~\eqref{eq:ATOS},   under Assumption~\ref{ass:function}, with      $\{\alpha^k\}$ satisfying \eqref{eq:adaptive stepsize selection rule},    for some  summable $\{n^k>0\}$. Then, for any given   $({\scr X}^*,{\scr S}^*)\in{\cal P}^*\times{\cal D}^*$,\vspace{-0.1cm}
\begin{equation}
    \label{eq:exact descent}
    \widetilde{{\cal V}}_{{\scr X}^*{\scr S}^*}^{k+1}\leq \widetilde{{\cal V}}_{{\scr X}^*{\scr S}^*}^k-\|\widetilde{\AF}^{k+1}-\widetilde{\X}^k\|^2-\frac{1-\delta}{2}\|\AF^{k+1}-\X^k\|^2-{\cal G}_{{\scr X}^*{\scr S}^*}^{k+1}. 
\end{equation}
 Moreover, $\lim_{k\rightarrow\infty} {\scr A}^{k+1}-{\scr X}^k=0$. Therefore, the sequence $\{{\cal V}_{{\scr X}^*{\scr S}^*}^k\}$ converges.
\end{lemma}

Notice that the rule \eqref{eq:adaptive stepsize selection rule}  produces {\it nonmonotone}, adaptive  stepsize sequences. 

\subsection{Decentralized implementation}
\label{subsec:decentralized}
The scheme in~\eqref{eq:ATOS},  coupled with the stepsize update~\eqref{eq:adaptive stepsize selection rule}, is not yet decentralized.  Indeed, the computation of \(\prox_{\alpha^k\widetilde{G}}\) and the quantities entering the stepsize rule require global information and thus cannot be carried out locally and independently by the agents.  In this section, we address these two obstacles and derive a fully decentralized implementation of~\eqref{eq:ATOS}.

We begin decentralizing the computation of  $\texttt{prox}_{\alpha^k\widetilde{G}}$, leveraging the degree of freedom offered by the BCV rule ($M$ matrix). 
\begin{lemma}[\cite{ryu2022large}]
    \label{lemma:composite G}
    Let $g(u)=f^*(A^\top u)$, with  $f:\mathbb{R}^n\to \mathbb{R}\cup\{-\infty, \infty\}$ assumed to be  convex, closed and proper. Suppose,  $\texttt{ri}\,\texttt{dom}\, f^*\cap \texttt{range}(A^\top)\neq\emptyset$. Then $v=\texttt{prox}_{\alpha,g}(u)$ if and only if there exists $x\in \mathbb{R}^n$ such that \vspace{-0.2cm}
\begin{equation}\label{eq:sub-conj}
        x\in\textit{argmin}_y\{f(y)-\langle u,Ay\rangle+\frac{\alpha}{2}\|Ay\|^2\}\quad \text{and}  
        \quad v=u-\alpha Ax.\vspace{-0.1cm}
\end{equation}
\end{lemma}

Notice that   $\widetilde{G}=(\delta_{\{0\}}^*)^*\circ [\PL,M]$, with $\delta_{\{0\}}^*$ denoting the conjugate of $\delta_{\{0\}}$.  Since  $\delta_{\{0\}}^*$ is convex, closed and proper    and $\texttt{ri}\texttt{dom} \,\delta_{\{0\}}\cap \texttt{Range}([\PL,M]])\neq\emptyset$,  we can   use \eqref{eq:sub-conj} to express $\texttt{prox}_{\alpha^k\widetilde{G}}$. Hence,    \eqref{eq:ATOS A} can be rewritten as   
\begin{equation}
    \label{eq:A}
{\scr A}^{k+1}=\begin{bmatrix}
    \AF^{k+1}\\\widetilde{\AF}^{k+1}
\end{bmatrix}=\begin{bmatrix}
        \X^k-\alpha^k\Sub^k -\alpha\nabla F(\X^k)\\
        \widetilde{\X}^k-\alpha^k\widetilde{\Sub}^k  
    \end{bmatrix}-\alpha^k\begin{bmatrix}
        \PL\\ M
    \end{bmatrix}\Y^{k+1}, 
\end{equation}
where \vspace{-0.3cm}\begin{equation}
    \label{eq:BCV Y 1'}
    \begin{aligned}       \Y^{k+1}\in\,&\texttt{argmin}_{\Y}\left\{\delta_{\{0\}}^*(\Y)-\langle \widetilde{\X}^k-\alpha^k\widetilde{\Sub}^k,M\Y\rangle\right.\\&-\langle \X^k-\alpha^k\Sub^k-\alpha^k\nabla F(\widetilde{\X}^k),\PL \Y\rangle+\frac{\alpha^k}{2}\left(\|\PL \Y\|^2+\|M \Y\|^2\right)\}.
    \end{aligned}
\end{equation}

We proceed choosing $M$ so that  (\ref{eq:BCV Y 1'}) admits a closed form solution, computable locally at the agents' sides.  
Combining~\eqref{eq:ATOS},~\eqref{eq:A}, and (see \eqref{eq:ATOS})
\begin{equation}
    \label{eq:BCV X 1'}
     \X^{k+1}=\texttt{prox}_{\alpha^kR}(\AF^{k+1}+\alpha^k\Sub^k)\quad \text{and}\quad \widetilde{\X}^{k+1}=\mathbf{0},
\end{equation}
 we have 
  $   \widetilde{\Sub}^k=-M\Y^{k}$,  $k\geq 1$.
Substituting $\widetilde{\Sub}^k$  in~\eqref{eq:BCV Y 1'}  and using   $\delta^*_{\{0\}}(\Y)=0$, yields
\begin{equation*}
 \begin{aligned}     \!\!\!\!\!   \Y^{k+1}\!\!\in
   \texttt{argmin}_\Y\,\left\{ \alpha^k \|\Y-\Y^k\|^2_{\PL^2+M^2}\!-2\langle\PL(\X^k-\!\alpha^k\Sub^k-\alpha^{k}\PL\Y^k-\alpha^k\nabla F(\X^k)),\Y\rangle\right\}.
 \end{aligned}
\end{equation*}
Choosing $M:=\sqrt{I-\PL^2}$,   $Y^{k+1}$ admits a closed form expression: 
\begin{equation}
\label{eq:BCV Y 3'}
\begin{aligned}
\Y^{k+1}\!\!=& \Y^k+\frac{1}{\alpha^k}\PL(\X^k\!-\!\alpha^k\Sub^k\!-\!\alpha^{k}\PL\Y^k-\alpha^k\nabla F(\X^k)).
\end{aligned}
\end{equation}
Introducing    $\D^k:=\PL \Y^k$, and setting  $\PL^2=I-W$, with $W=(1-c)I+c\widetilde{W}\in W_{\mathcal{G}}$ and $c\in(0,1/2)$,~\eqref{eq:BCV Y 3'} can be rewritten in terms of the  $\D$-variables as
\begin{equation}
\label{eq:D}
    \D^{k+1}=\D^k+\frac{1}{\alpha^k}(I-W)(\X^k-\alpha^k(\Sub^k+\D^k+\nabla F(\X^k))),
\end{equation}
and so the update of the    
 ${\scr A}$-variables (using \eqref{eq:A}):  
\begin{equation}\label{eq:TA-final}
    \AF^{k+1}=W\X^k-\alpha^k W(\nabla F(\X^k)+\Sub^k+\D^k)).
\end{equation}
Finally, from \eqref{eq:ATOS}, we directly obtain the update of the $\Sub$-variables. 

In summary, the proposed  algorithm is given by   (\ref{eq:BCV X 1'}),  (\ref{eq:D}),   (\ref{eq:TA-final}), and \eqref{eq:ATOS}.  The algorithm is now fully decentralized in the updates of $\X$-, $\Sub$- and $\D$-variables.   


We are left to address the implementability of the adaptive stepsize selection rule~\eqref{eq:adaptive stepsize selection rule} over the network. To make   $\|{\scr S}^k-{\scr S}^0\|^2=\|\Sub^k-\Sub^0\|^2+\|\widetilde{\Sub}^k-\widetilde{\Sub}^0\|^2$  computable from the agents, we need a decentralized way to evaluate (or upper bound) the \emph{slack} term
\(\|\widetilde{\Sub}^k-\widetilde{\Sub}^0\|^2\). To this end, we exploit the fact that ~\eqref{eq:BCV Y 3'} evolves \(\Y^k\) through increments lying in \(\texttt{range}(\PL)\). We therefore introduce a   tracking variable $\T^k\in\mathbb{R}^{m\times d}$  (initialized at $\T^0=0$) via  \vspace{-0.2cm}
\begin{equation*}
    \T^{k+1}=\T^{k}+\frac{1}{\alpha^{k}}\X^{k}-\Sub^{k}-\D^{k}-\nabla F(\X^{k}),\vspace{-0.2cm}
\end{equation*}so that by construction and induction $\PL\T^k=\Y^k-\Y^0$. Recalling that \(\widetilde{\Sub}^k=-M\Y^k\), we obtain
$\widetilde{\Sub}^k-\widetilde{\Sub}^0=-M\text{\PL}\T^k.$ 
  Consequently, 
\begin{align}
    \label{eq:upper bound est}
  \|{\scr S}^k-{\scr S}^0\|^2= \|\Sub^k-\Sub^0\|^2+\|\widetilde{\Sub}^k-\widetilde{\Sub}^0\|^2 \nonumber \leq    \|\Sub^k-\Sub^0\|^2+  {2c}\|\T^k\|^2.
\end{align}
This bound   makes the stepsize rule implementable without   maintaining \(\widetilde{\Sub}^k\) (or \(\Y^k\)).

To satisfy the first condition in \eqref{eq:adaptive stepsize selection rule}, we exploit the additive separability of $L^k$ (see \eqref{eq:Lk}), yielding   the following sufficient condition for $0<\alpha^k\leq \delta/L^k$: for each agent $i$, find the largest $\alpha_i^k>0$ satisfying the following {\it local} backtracking condition:
\begin{equation}
    \label{eq:line search}
    f_i(a_i^{k+1})\leq  f_i(x_i^{k})+\langle \nabla f_i(x_i^{k}),a_i^{k+1}-x_i^{k}\rangle+\frac{\delta}{2 {\alpha}_i^k}\|a_i^{k+1}-x_i^{k}\|^2,
\end{equation}
and set $\alpha_i^k\leftarrow \alpha^k:=\min_{i\in[m]}{\alpha}_i^k$. In \eqref{eq:line search}, $a_i^k$, $d_i^k$, $s_i^k$, $x_i^k$ and $t_i^k$ are the $i-$th row of $\AF^k$, $\D^k$, $\Sub^k$, $\X^k$ and $\T^k$, respectively. 
 
 

The proposed decentralized algorithm combining all the steps above is summarized in  Algorithm~\ref{alg:DATOS}, with the   backtracking procedure provided in Algorithm~\ref{alg:backtracking_DATOS}.  

\begin{algorithm}[ht!]
\centering
\resizebox{\columnwidth}{!}{%
  \begin{minipage}{\columnwidth}
\caption{Decentralized Adaptive Three Operator Splitting}
  \noindent \textbf{Data:} (i) Initialization:  $\alpha^{-1} \in (0, \infty)$, $\X^0,\Sub^0 \in \mathbb{R}^{m\times n}$  and  $\X^{-1},\AF^0,\D^0,\T^0 = \mathbf{0}$;  (ii) Backtracking   parameters: summable sequence $\{n^k\}$, $\eta\in(0,1)$, $\delta \in (0,1)$;  (iii) Gossip matrix  $W=(1-c)I+c\widetilde{W}$, $c\in(0,1/2)$.		
		\begin{algorithmic}[1]
           \State \texttt{(S.1) Communication Step: } $$\X^{k+1/2}=W\X^k,\quad \D^{k+1/2}=W(\nabla F(\X^k)+\Sub^k+\D^k);$$
          
			\State \texttt{(S.2) Decentralized line-search: }Each agent updates ${\alpha}_i^k$ according to:
 \begin{equation}
     \label{eq:largest stepsize}
     \widetilde{\alpha}_i^{k}=\sqrt{(\alpha_i^{k-1})^2+\min\left\{\frac{1-\delta}{4}\frac{\|a_i^{k}-x_i^{k-1}\|^2}{\|s_i^k-s_i^0\|^2+2c\|t_i^k\|^2},\, n^k\right\}},
 \end{equation}
    $$ {\alpha}_i^k=\texttt{Linesearch}( \widetilde{\alpha}_i^{k},x_i^k,x_i^{k+1/2},-d_i^{k+1/2},\eta,\delta);$$
			\State \texttt{(S.3)} \texttt{Global min-consensus: }
            $$\alpha^k_i\leftarrow \alpha^k:=\min_{i\in[m]} {\alpha}_i^k;$$
     \State \texttt{(S.4) Updates of the primal dual and tracking variables:}
   \begin{subequations}
   \label{eq:DATOS}
         \begin{align}
        \label{eq:DATOS A}
         \AF^{k+1}&=\X^{k+1/2}-\alpha^k\D^{k+1/2},\\
             \label{eq:DATOS X}
    \X^{k+1}&=\prox_{\alpha^k R}(\AF^{k+1}+\alpha^k\Sub^k),\\
    \label{eq:DATOS S}
    \Sub^{k+1}&=\Sub^k+\frac{1}{\alpha^k}(\AF^{k+1}-\X^{k+1}),\\       
    \label{eq:DATOS D}\D^{k+1}&=\D^{k+1/2}-\nabla F(\X^k)-\Sub^k+\frac{1}{\alpha^k}(\X^k-\X^{k+1/2}),\\
    \label{eq:DATOS T}
    \T^{k+1}&=\T^{k}-\Sub^{k}-\D^{k}-\nabla F(\X^{k})+\frac{1}{\alpha^{k}}\X^{k}.
     \end{align}
   \end{subequations}
			\State \texttt{(S.5)}    If a termination criterion is not met,  $k\leftarrow k+1$ and go to step \texttt{(S.1)}.
\end{algorithmic}\label{alg:DATOS}
 \end{minipage}
}
\end{algorithm}
\begin{algorithm}\centering
\resizebox{\columnwidth}{!}{%
  \begin{minipage}{\columnwidth}
\caption{\texttt{Linesearch}($\alpha,x_1,x_2,d,\eta,\delta$)}
		\begin{algorithmic}[1]
        \State $x^+:=x_2+\alpha d$; \texttt{set} $t=1$;
           \While{$f(x^+)> f(x_1)+\langle\nabla f(x_1),x^+-x_1\rangle+\frac{\delta}{2\alpha}\|x^+-x_1\|^2$}
			\State $\alpha\leftarrow \eta\alpha$;
            \State $x^+\leftarrow x_2+\alpha d$;
            \State $t\leftarrow t+1$;
            \EndWhile
            \Return $\alpha$.
\end{algorithmic}\label{alg:backtracking_DATOS}
\end{minipage}}
\end{algorithm}
 \noindent\textit{On the global-min consensus:}     
 Step \texttt{(S.3)} in Algorithm~\ref{alg:DATOS} runs a network-wide min-consensus to enforce a common stepsize,
\(\alpha^k=\min_{i\in[m]}\alpha_i^k\).
This operation is compatible with modern wireless mesh networks that support multiple interfaces, e.g., WiFi and LoRa (Long- Range)~\cite{Kim_Lim_Kim_2016,Janssen_BniLam_Aernouts_Berkvens_Weyn_2020}: WiFi, which provides high-speed, short-range communication capabilities, can be used for exchanging vector variables in Step \texttt{(S.1)}; and LoRa--supporting long-range communication over low rates--can broadcast the scalar \(\alpha_i^k\) to all nodes (possibly after coarse quantization) in a single communication to implement Step \texttt{(S.3)}. A purely neighbor-based alternative relying on local min-consensus is developed in Sec.~\ref{sec:local min}.\vspace{-0.3cm}

\subsection{Proof of Lemma~\ref{lemma:distance characterization}}
\label{subsec:proof of lemma 1}
\begin{proof}
By~\eqref{eq:ATOS A}, it holds that\vspace{-0.2cm}
\begin{equation}
    \label{eq:subgradient G}
    \frac{1}{\alpha^k}\left({\scr X}^k-\alpha^k{\scr S}^k-\alpha^k\nabla \widetilde{F}({\scr X}^k)-{\scr A}^{k+1}\right)\in\partial\widetilde{G}({\scr A}^{k+1}).
\end{equation}
Invoking convexity of $\widetilde{F}$ and $\widetilde{G}$ yields
\begin{equation}
    \label{eq:primal distance 1}
    \begin{aligned}
        \|{\scr X}^{k+1}-{\scr X}\|^2
        =\,&\|{\scr X}^{k}-{\scr X}\|^2-\|{\scr A}^{k+1}-{\scr X}^{k}\|^2+\|{\scr X}^{k+1}-{\scr A}^{k+1}\|^2\\
        &\underbrace{-2\langle {\scr X}^{k}-{\scr A}^{k+1},{\scr A}^{k+1}-{\scr X}\rangle+2\langle {\scr X}^{k+1}-{\scr A}^{k+1},{\scr A}^{k+1}-{\scr X}\rangle}_{\texttt{term I}}
    \end{aligned}
\end{equation}
Replacing the ${\scr X}-$update~\eqref{eq:ATOS} in the second part of $\texttt{term I}$, we have
\begin{equation}
    \label{eq:term I distance}
\begin{aligned}
    \texttt{term I}
        \overset{\eqref{eq:subgradient G}}{\leq}&-2\alpha^k(\widetilde{G}({\scr A}^{k+1})-\widetilde{G}({\scr X})-\langle{\scr S}^{k+1},{\scr A}^{k+1}-{\scr X}\rangle)\\
        &-2\alpha^k(\widetilde{F}({\scr X}^k)-\widetilde{F}({\scr X}))-2\alpha^k\langle \nabla\widetilde{F}({\scr X}^k),{\scr A}^{k+1}-{\scr X}^k\rangle\\
        \overset{\eqref{eq:new Lagarangian}}{=}&-2\alpha^k\left(\mathcal{L}({\scr A}^{k+1},{\scr S}^{k+1})-\mathcal{L}({\scr X},{\scr S}^{k+1})\right)\\
        &+\underbrace{2\alpha^k(\widetilde{F}({\scr A}^{k+1})-\widetilde{F}({\scr X}^{k})-\langle \nabla\widetilde{F}({\scr X}^k),{\scr A}^{k+1}-{\scr X}^k\rangle)}_{\texttt{term II}}.
\end{aligned}
\end{equation}
By the lifted structure of $\widetilde{F}$ and definition of $L^k$ in~\eqref{eq:Lk}, it follows 
   $    \texttt{term II}=2\alpha^k(F(\AF^{k+1})-F(\X^{k})-\langle \nabla F(\X^k),\AF^{k+1}-\X^k\rangle)=L^k\alpha^k\|\AF^{k+1}-\X^k\|^2.$
   
 Using the above expression along with \eqref{eq:term I distance} in \eqref{eq:primal distance 1}, leads to~\eqref{eq:primal distance}. 
 
  We prove now~\eqref{eq:dual distance}.  By~\eqref{eq:ATOS}, we have  \vspace{-.2cm} 
   \begin{equation}
    \label{eq:subgradient R*}
        {\scr S}^{k+1}\in\partial \widetilde{R}({\scr X}^{k+1})\Rightarrow {\scr X}^{k+1}\in\partial\widetilde{R}^*({\scr S}^{k+1}). \vspace{-.2cm}
    \end{equation}
 By~\eqref{eq:ATOS} and  convexity of $\widetilde{R}^*$,  \vspace{-0.2cm}
    \begin{align*}
        \|{\scr S}^{k+1}-{\scr S}\|^2
        = &\,\|{\scr S}^k-{\scr S}\|^2-\left(\frac{1}{\alpha^k}\right)^2\|{\scr A}^{k+1}-{\scr X}^{k+1}\|^2\\
        &+2\langle{\scr S}^{k+1}-{\scr S},\frac{1}{\alpha^k}({\scr A}^{k+1}-{\scr X}^{k+1})\rangle\\
        \overset{\eqref{eq:subgradient R*}}{\leq}&\|{\scr S}^k-{\scr S}\|^2-\left(\frac{1}{\alpha^k}\right)^2\|{\scr A}^{k+1}-{\scr X}^{k+1}\|^2\\
        &+\frac{2}{\alpha^k}\left(\langle {\scr S}^{k+1}-{\scr S},{\scr A}^{k+1}\rangle -\left(\widetilde{R}^*({\scr S}^{k+1})-\widetilde{R}^*({\scr S})\right)\right)\\
        \overset{\eqref{eq:new Lagarangian}}{=}&\|{\scr S}^k-{\scr S}\|^2-\left(\frac{1}{\alpha^k}\right)^2\|{\scr A}^{k+1}-{\scr X}^{k+1}\|^2\\
        &+\frac{2}{\alpha^k}\left(\mathcal{L}({\scr A}^{k+1},{\scr S}^{k+1})-\mathcal{L}({\scr A}^{k+1},{\scr S})\right).
    \end{align*}
\end{proof}

\section{Convergence Analysis of Algorithm~\ref{alg:DATOS}}
\label{sec:convergence cvx}
For analytical convenience, we will work with the stacked formulation~\eqref{eq:ATOS}. 
Since \(\D^0=0\in \texttt{span}(I-W)\), Algorithm~\ref{alg:DATOS} is equivalent to~\eqref{eq:ATOS}, provided that the stepsize \(\alpha^k\) is updated according to the backtracking rule in \texttt{(S.3)}. Henceforth, we will freely refer to~\eqref{eq:ATOS} while implicitly invoking this equivalence.    Accordingly,  we will use $\{\widetilde{\cal V}^k_{{\scr X}^*,{\scr S}^*}\}$  as defined in (\ref{eq:companion-Lyapunov})--with $\{{\cal V}^k_{{\scr X}^*,{\scr S}^*}\}$ given by \eqref{eq:Lyapunov_0}--to be evaluated along the iterates generated by Algorithm~\ref{alg:DATOS} (with a slight abuse of notation).

\subsection{Asymptotic convergence} Building on Lemma~\ref{lemma:lyapunov asymptotic}, we begin establishing   boundedness   of the iterates and the stepsizes generated by Algorithm~\ref{alg:backtracking_DATOS}. 

\begin{lemma}
    \label{lemma:legal stepsize}
Let  $\{({\scr A}^{k},{\scr X}^{k},{\scr S}^{k})\}$ be the sequence generated  by Algorithm~\ref{alg:DATOS} under Assumption~\ref{ass:function}, and  some summable  $\{n^k>0\}$ in (\ref{eq:largest stepsize});    the following hold:
    \begin{itemize}
        \item[\textbf{(i)}] For any $k\geq 0$,\vspace{-0.2cm}
        \begin{equation}
            \label{eq:descent inequality sec 3}\widetilde{{\cal V}}_{{\scr X}^*{\scr S}^*}^{k+1}\leq \widetilde{{\cal V}}_{{\scr X}^*{\scr S}^*}^k-\|\widetilde{\AF}^{k+1}-\widetilde{\X}^k\|^2-\frac{1-\delta}{2}\|\AF^{k+1}-\X^k\|^2-{\cal G}_{{\scr X}^*{\scr S}^*}^{k+1},\vspace{-0.2cm}
        \end{equation}
       and $\lim_{k\rightarrow\infty}{\scr A}^{k+1}-{\scr X}^k=0$; hence   $\{{\cal V}^k_{{\scr X}^*{\scr S}^*}\}$  convergences; and    $\{{\scr X}^k\}_{k\geq0}$,    $\{\alpha^{k-1}\scr S^k\}_{k\geq 0}$  and  $\{{\scr A}^k\}_{k\geq 0}$ are bounded. Consequently, the  set ${\cal H}:=\texttt{conv}$ $\left(\cup_{k=0}^{\infty}[\X^k, \AF^{k+1}]\right)$ is compact and there 
         exists $L_f\in(0,\infty)$ such that $F$ is $L_f$-smooth  on   ${\cal H}$. Furthermore, \vspace{-0.1cm}
        \begin{equation}
            \label{eq:stepsize lower bound}
            \alpha^k\geq \underline{\alpha}:=\min\left\{\alpha^{-1},\frac{\eta\delta}{L_f}\right\}>0,\quad \forall k\geq 0,\vspace{-0.1cm}
        \end{equation}
      and thus $\{{\scr S}^k\}_{k\ge0}$ is bounded;
  \item[\textbf{(ii)}]  
  $\{\alpha^k\}$ converges and
 $\lim_{k\to\infty}\alpha^k=\alpha\ge \underline{\alpha}$. 
    \end{itemize} 
\end{lemma}
\begin{proof}
\textbf{(i)} Notice that  the stepsize sequence, as generated   from  (\texttt{S.2}) and (\texttt{S.3}) in Algorithm~\ref{alg:DATOS}, still satisfies the  bounds (\ref{eq:adaptive stepsize selection rule}). Therefore the conclusions of  Lemma~\ref{lemma:lyapunov asymptotic} are applicable also to  Algorithm~\ref{alg:DATOS}. This proves    \eqref{eq:descent inequality sec 3},  convergence of $\{\widetilde{\cal V}_{{\scr X}^* {\scr S}^*}\}$, and    $\lim_{k\rightarrow\infty}{\scr A}^{k+1}-{\scr X}^k=0$. Consequentially,   $\{{\cal V}^k_{{\scr X}^*{\scr S}^*}\}$ also convergence. 

Convergence of $\{{\cal V}^k_{{\scr X}^*{\scr S}^*}\}$ implies  the  boundedness of $\{{\scr X}^k\}$ and $\{\alpha^{k-1}{\scr S}^k\}$; and, from    \(\|{\scr A}^{k+1}-{\scr X}^k\|\to 0\), also that of $\{{\scr A}^k\}$. Therefore ${\cal H}$ is compact and \(F\) is  smooth on \({\cal H}\). The line-search then guarantees the uniform lower bound~\eqref{eq:stepsize lower bound}. Finally, boundedness of \(\{\alpha^{k-1}{\scr S}^k\}\) together with \(\alpha^{k-1}\ge \underline{\alpha}\) implies boundedness of \(\{{\scr S}^k\}\).  



\textbf{(ii)}
Using   $(\alpha^k)^2\leq (\alpha^{k-1})^2+n^k$, for all $k\geq 0$, yields  \((\alpha^k)^2\le (\alpha^0)^2+\sum_{t=1}^k n^t \leq \bar\alpha\), for some \(\bar\alpha<\infty\), due to the summability of $\{n^k\}$. 
Therefore, \(\underline{\alpha}\le \alpha^k\le \bar\alpha\). 
We next show that the line-search  decreases \(\alpha^k\) only finitely many times.   

If a   strict decrease occurs at iteration $k\geq 1$,  $\alpha^{k}< \alpha^{k-1}$, it must be
\begin{equation}
    \label{eq:descent stepsize}
    \alpha^k\leq \eta\sqrt{(\alpha^{k-1})^2+n^k}.
\end{equation}
Since $ \sum_{t=0}^{\infty}n^t<\infty$,  there exists $K\geq 1$ such  that \vspace{-.1cm}
\begin{equation} 
\label{eq:never come back}
    \sum_{t=k}^{\infty}n^t<\frac{1-\eta^2}{2}(\underline{\alpha})^2,\quad \forall k\geq K.\vspace{-.1cm}
\end{equation} 
Suppose a strict decrease happens at $k\geq K$. Then for any $k'\geq k$, it  holds  
$$  (\alpha^{k'})^2\leq \eta^2\left[(\alpha^{k-1})^2+n^{k}\right]+\sum_{t=k+1}^{k'}n^t
    \overset{\eqref{eq:never come back}}{\leq}
\left(\eta^2+\frac{1-\eta^2}{2}\right)(\alpha^{k-1})^2<(\alpha^{k-1})^2.
 $$
Thus after $k\geq K$, every strict decrease at   $k$ maintains $\alpha^{k'}<\alpha^{k-1}$ for all subsequent $k'\geq k$.    Since \(\alpha^k\geq \underline{\alpha}>0\) for all \(k\), these  decrease events can occur only finitely many times;
hence, \(\{\alpha^k\}\) is eventually nondecreasing; being also bounded, it converges.
   \end{proof}

We are now ready to state the asymptotic convergence of  Algorithm~\ref{alg:DATOS}.
\begin{theorem}
    \label{thm:iterates convergence}
    Under the conditions of Lemma~\ref{lemma:legal stepsize},   
   the iterates  $\{({\scr X}^k,{\scr S}^k)\}_{k\geq 0}$ converge  to 
    to some $({\scr X}^*,{\scr S}^*)\in{\cal P}^*\times {\cal D}^*$. Moreover, letting  $x^*$ being any minimizer of Problem~\eqref{eq:problem}, it holds   $\lim_{k\rightarrow\infty}\X^k=\X^*:={1}(x^*)^\top$.
\end{theorem}
\begin{proof}

By Lemma~\ref{lemma:legal stepsize}(i), the sequences \(\{{\scr X}^k\}\), \(\{{\scr S}^k\}\) (and \(\{{\scr A}^k\}\)) are bounded; hence there exists a subsequence
\(\{N_k\}\subset\mathbb{N}_+\) such that
$
({\scr X}^{N_k},{\scr S}^{N_k})\to({\scr X}^*,{\scr S}^*),
$
and  \(\alpha^{N_k}\to\alpha\). 
Further, since $\lim_{k\rightarrow\infty}{\scr A}^{k+1}-{\scr X}^k$=0, it can be derived that the sequence $\{({\scr X}^{N_k+1},{\scr S}^{N_k+1})\}$ also converge to $({\scr X}^*,{\scr S}^*)$. Since \(({\scr X}^{k+1},{\scr S}^{k+1})=\mathbb{T}_{\alpha^{k}}({\scr X}^{k},{\scr S}^{k})\)    and \(\mathbb{T}_\alpha({\scr X},{\scr S})\) is continuous in \(({\scr X},{\scr S},\alpha)\),
 passing to the limit (along $\{N_k\}$) yields \(({\scr X}^*,{\scr S}^*)=\mathbb{T}_\alpha({\scr X}^*,{\scr S}^*)\).
  Following similar steps as in  \cite[Lemma 2.2]{davis2017three}, one can show that  $({\scr X}^*,{\scr S}^*)\in \mathcal P^\star\times  \mathcal D^\star$.    

We now prove convergence of the whole sequence. By  Lemma~\ref{lemma:lyapunov asymptotic}, for any saddle point \(({\scr X}^*,{\scr S}^*)\) the merit sequence \({\cal V}^k_{{\scr X}^*{\scr S}^*}\) in (\ref{eq:Lyapunov_0}) convergences (to some finite limit).  On the other hand, along the subsequence \(\{N_k\}\) we have \(({\scr X}^{N_k},{\scr S}^{N_k})\to({\scr X}^*,{\scr S}^*)\),
so \({\cal V}^{N_k}_{{\scr X}^*{\scr S}^*}\to 0\). Therefore, it must be \({\cal V}^{k}_{{\scr X}^*{\scr S}^*}\to 0\).  This implies  \({\scr X}^k\to{\scr X}^*\). Moreover, since \(\alpha^{k-1}\ge \underline{\alpha}>0\), we also obtain
\(\|{\scr S}^k-{\scr S}^*\|^2\le \underline{\alpha}^{-2}{\cal V}^k_{{\scr X}^*{\scr S}^*}\to 0\), i.e., \({\scr S}^k\to{\scr S}^*\).
Hence \(({\scr X}^k,{\scr S}^k)\to({\scr X}^*,{\scr S}^*)\).

From $({\scr X}^*,{\scr S}^*)\in \mathcal P^\star\times  \mathcal D^\star$, \({\scr X}^*\) is a primal optimal solution of~\eqref{eq:problem}; thus \({\scr X}^*={\bf 1}(x^*)^\top\) for some minimizer \(x^*\),
and consequently \(\lim_{k\to\infty}\X^k=\X^* =  1(x^*)^\top\).
\end{proof}

\paragraph{On the design of the summable  sequence \texorpdfstring{$\{n^k\}$}{n^k}} The sequence \(\{n^k\}\) controls the amount of increase of  the stepsize   between two consecutive iterations.  
For convergence  \{\(n^k\}\) is requested to be summable.  This naturally suggests the following    rule \vspace{-0.1cm}
\begin{equation}\label{eq:nk_optionB_correct}
n^k  = \frac{\beta}{(k+1)^p},\qquad \beta>0,\ \ p>1.\vspace{-0.1cm}
\end{equation}
This choice   may be  conservative when occasional backtracking events make \(\alpha^k\) drop. 
To be   {\it adaptive},  
  \(\{n^k\}\)   should not decay solely because \(k\) grows, especially during phases in which \(\alpha^k\) is forced to remain small due to a large local curvature estimate \(L^k\).   
  
  Next we propose a rule that  ``restarts'' $n^k$  after a 
significant   drop  of the stepsize.  More formally,   for fixed   \(\eta'\in(\eta,1)\),  define the set of ``drop times'' up to \(k\) as
\begin{equation}\label{eq:drop_times}
\mathcal{K}^{k}:=\Big\{j\in\{0,\ldots,k\}:\ \alpha^{j}\le \eta'\min_{0\le t\le j-1}\alpha^t\Big\},
\end{equation}
with the convention $\min_{0\le t\le -1}\alpha^t=\alpha^{-1}$    
Define the \emph{reset counter} and \emph{time since last reset} respectively as\vspace{-0.1cm}
\[
r^k := |\mathcal K^{k}|
\quad\text{and}\quad 
\tau_r^k := k-\max  \mathcal{K}^{k},
\]
with  the convention $\max\emptyset=-1$ (so if no drop occurs up to $k$, then $\tau_r^k=k+1$).  Fix $\beta>0$ and exponents $p>1$, $q>1$. Define, for all $k\ge 0$,
\begin{equation}\label{eq:B_summable_rule}
n^k := \frac{\beta}{(r^k+1)^q}\cdot \frac{1}{(\tau_r^k+1)^p}.
\end{equation}
 The  rationale of the above rule is the following: $\tau_r^k$ grows while no new drop occurs (so the decay persists during ``stalls''),
whereas $r^k$ grows with the number of drops (so the total remains summable even if drops happen
often).
\begin{lemma}
    The sequence $\{n^k\}$ defined in (\ref{eq:B_summable_rule}) is summable, for every   sequence $\{\alpha^k>0\}$,  $\beta>0$, and $p,q>1$. Specifically, letting $S_p:=\sum_{t=1}^{\infty} t^{-p}<\infty$ and $S_q:=\sum_{t=1}^{\infty} t^{-q}<\infty$, it holds  $\sum_{k=0}^{\infty} n^k \le \beta\, S_p\, S_q.$

\end{lemma}

\subsection{Complexity analysis}
\label{subsec:sublinear}
Next we establish the convergence rate      of Algorithm~\ref{alg:DATOS}. 
For any given saddle-point    $({\scr X}^*,{\scr S}^*)$  of   $\mathcal L$, let  $\mathcal{B}_X\times \mathcal{B}_S$ be a bounded set containing $({\scr X}^*,{\scr S}^*)$; and let us   introduce the partial primal-dual gap function~\cite{chambolle2011first}: 
\begin{equation}
    \label{eq:new lyapunov}
    \mathcal{G}_{\mathcal{B}_X\times \mathcal{B}_S}({\scr X},{\scr S}) :=\max_{{\scr S}^\prime\in B_S}\mathcal{L}({\scr X},{\scr S}^\prime)-\min_{{\scr X}^\prime\in\mathcal{B}_X}\mathcal{L}({\scr X}',{\scr S}).\vspace{-0.1cm}
\end{equation} 
Notice that   $\mathcal{G}_{\mathcal{B}_X\times \mathcal{B}_S}({\scr X},{\scr S})\geq 0$, for any  $({\scr X},{\scr S})\in \texttt{dom}(\mathcal L)$. Yet, $\mathcal{G}_{\mathcal{B}_X\times \mathcal{B}_S}$ is not a valid measure of optimality because it might vanish for $({\scr X},{\scr S})$ not being a saddle-point of $\mathcal L$ \cite{chambolle2011first}.  However, if $\mathcal{G}_{\mathcal{B}_X\times \mathcal{B}_S}({\scr X},{\scr S})=0$ for   $({\scr X},{\scr S})\in \texttt{dom}(\mathcal L)$  and, in addition, $({\scr X},{\scr S})$ lies in the interior of $\mathcal{B}_X\times \mathcal{B}_S$, then $({\scr X},{\scr S})$ is a saddle-point of $\mathcal L$. 
Since  $({\scr X}^k,{\scr A}^k,{\scr S}^k)$ generated by Algorithm~\ref{alg:DATOS} is bounded (Lemma~\ref{lemma:legal stepsize}),   one can always choose a ``sufficiently large''  $\mathcal{B}_X\times \mathcal{B}_S$ such that $\mathcal{G}_{\mathcal{B}_X\times \mathcal{B}_S}({\scr X},{\scr S})$ is a valid measure of 
optimality for $({\scr X},{\scr S})$. 
We provide next the convergence rate decay of $\mathcal{G}_{\mathcal{B}_X\times \mathcal{B}_S}$.

Define  the    ergodic sequences   of the   iterates     $
    \{{\scr A}^k\}$ and  $\{{\scr S}^k\}$:\vspace{-0.2cm}
 \begin{equation}\label{eq:weighted_average}\overline{{\scr A}}^k:=\frac{1}{\theta^k}\sum_{t=1}^{k}\alpha^{t-1}{\scr A}^{t} \,\,\text{ and }\,\, \overline{{\scr S}}^{k}:=\frac{1}{\theta^k}\sum_{t=1}^{k}\alpha^{t-1}{\scr S}^{t},\,\,\text{ with }\,\,\theta^k:=\sum_{t=1}^{k}\alpha^{t-1}.\vspace{-0.3cm}\end{equation}
\begin{theorem}
\label{thm:asymptotic convergence} Under the assumptions of Lemma~\ref{lemma:legal stepsize}, fix any saddle point
$({\scr X}^*,$ ${\scr S}^*)\in \mathcal P^\star\times \mathcal D^\star$ and consider the set 
   \begin{equation}
       \label{eq:set_Bs}
\mathcal{B}_{X}\times\mathcal{B}_{S}:=\left\{({\scr X},{\scr S}):\|{\scr X}-{\scr X}^*\|^2+(\underline{\alpha})^2\|{\scr S}-{\scr S}^*\|^2\leq 10\, {\cal R}_{{\scr X}^*{\scr S}^*}\right\}, 
   \end{equation}
     where ${\cal R}_{{\scr X}^*{\scr S}^*}$ is defined in~\eqref{eq:bounded set diameter} below.
   Then:  \textbf{(i)} $({\scr X}^*,{\scr S}^*)\in \mathcal{B}_{X}\times\mathcal{B}_{S}$; \textbf{(ii)} $({\scr A}^{k+1},{\scr S}^{k+1})\in \texttt{int}(\mathcal{B}_{X}\times\mathcal{B}_{S})$, for all $k\geq 0$; and
 \textbf{(iii)} for all $k\geq 1$, \begin{equation}
        \label{eq:sublinea DATOS}
\mathcal{G}_{\mathcal{B}_X\times \mathcal{B}_S}(\overline{{\scr A}}^k,\overline{{\scr S}}^k)\leq\frac{1}{k}\cdot\frac{\overline{{\cal R}}_{{\scr X}^*{\scr S}^*}}{ \min\left(\alpha^{-1},(\eta\delta)/{L}_{\nabla}^k\right)}=\mathcal{O}\left(
\frac{1}{k}\right),\end{equation}
where ${L}_{\nabla}^k(\leq L_f)$ is the Lipschitz constant of $\nabla F$ over  ${\cal H}^k=\texttt{conv}\left(\cup_{t=0}^{k-1}[\X^t,\AF^{t+1}]\right)$ and $\overline{{\cal R}}_{{\scr X}^*{\scr S}^*}$ is defined in~\eqref{eq:sublinear DATOS pre}. 
\end{theorem}
\paragraph{Shorthands used in Thm.~\ref{thm:asymptotic convergence}.}
Define\vspace{-.2cm}
\begin{subequations}
    \begin{equation}
\label{eq:bounded set diameter}
{\cal R}_{{\scr X}^* {\scr S}^*}
:=
{\cal V}^0_{{\scr X}^* {\scr S}^*}
+\|{\scr S}^0-{\scr S}^*\|^2\sum_{t=0}^{\infty}n^t
\;<\;\infty,\quad \text{and set }\vspace{-.5cm}
\end{equation}
\begin{equation}
\label{eq:sublinear DATOS pre}
\overline{{\cal R}}_{{\scr X}^*,{\scr S}^*}
:=c_\alpha\,{\cal R}_{{\scr X}^*{\scr S}^*},\quad \text{with}\quad 
c_\alpha:=5\,\frac{(\alpha^{-1})^2+\sum_{t=0}^{\infty}n^t}{(\underline{\alpha})^2}.\vspace{-0.3cm}
\end{equation}\end{subequations}
\begin{proof} From~\eqref{eq:descent inequality sec 3} in Lemma~\ref{lemma:legal stepsize}(\textbf{i})   and $\AF^0=\X^{-1}=\mathbf 0$, we have
${\cal V}^k_{{\scr X}^*{\scr S}^*}\le {\cal R}_{{\scr X}^*{\scr S}^*}$, for all $k\ge0$. Item \textbf{(i)} is immediate from~\eqref{eq:set_Bs}.  
By~\eqref{eq:ATOS},  
it holds  $({\scr A}^{k+1},{\scr S}^{k+1})\in\texttt{int}(\mathcal{B}_{X}\times\mathcal{B}_{S})$;  hence ${\cal G}_{{\cal B}_X\times {\cal B}_S}(\overline{{\scr A}}^k,\overline{{\scr S}}^k)$ is a valid measure of optimality and progress of the algorithm. 
Replacing $({\scr X}^*,{\scr S}^*)$ in~\eqref{eq:descent inequality sec 3} with any $({\scr X}',{\scr S}')$ and telescoping, yield: for any $k\geq 1$, \vspace{-0.2cm}
\begin{equation}
    \label{eq:telescope}
2\sum_{t=1}^k\alpha^{t-1} ({\cal L}({\scr A}^t,{\scr S}')-{\cal L}({\scr X}',{\scr S}^t)) 
\leq   
{\cal V}^0_{{\scr X}',{\scr S}'} 
+\|{\scr S}^0-{\scr S}'\|^2\sum_{t=0}^{\infty}n^t.\vspace{-0.2cm}
\end{equation}
Dividing both sides of~\eqref{eq:telescope} by $2\theta_k$, leveraging convexity of ${\cal L}(\bullet,{\scr S}')-{\cal L}({\scr X}',\bullet)$, and taking the maximum of $({\scr X}',{\scr S}')$ over ${\cal B}_X\times {\cal B}_S$ on both   sides, yields $\mathcal{G}_{\mathcal{B}_X\times \mathcal{B}_S}(\overline{{\scr A}}^k,\overline{{\scr S}}^k)$ $\leq ({\overline{{\cal R}}_{{\scr X}^*{\scr S}^*}})/\theta^k$, 
where $\theta^k\geq\sum_{t=1}^k\min\big(\alpha^{-1},(\eta\delta)/\widetilde{L}^{t-1}\big)
\geq k\cdot \min(\alpha^{-1},(\eta\delta)/\widetilde{L}^k).$ 
\end{proof}
Theorem~\ref{thm:asymptotic convergence} proves that Algorithm~\ref{alg:DATOS} attains a sublinear $\mathcal{O}(1/k)$ rate, matching    state-of-art  algorithms for decentralized convex optimization  problems.  Crucially, it does so without requiring prior knowledge of optimization or network parameters, unlike existing native decentralized schemes for Problem~\eqref{eq:problem} (e.g., \cite{shi2015proximal,guo2023decentralized}). Moreover, the bound in~\eqref{eq:sublinea DATOS} is governed by  \textit{local} Lipschitz constants  $\{{L}^k_\nabla\}$, along the trajectory, which are often much smaller than the global constants used in non-adaptive methods—highlighting DATOS’ adaptation to the local geometry of $F$.\vspace{-0.3cm}

\section{From Global to Local Min-consensus}
\label{sec:local min}
In this section we introduce a variant of Algorithm~\ref{alg:DATOS} wherein the global min-consensus in   \texttt{(S.3)} is replaced by a local one--see Algorithm~\ref{alg:DATOS_local}.  
Step \texttt{(S.3)} therein   produces   different local stepsizes, collected in    $ \Lambda^k:=\texttt{diag}(\alpha_1^k,\alpha_2^k,\cdots,\alpha_m^k)$, and requires only {\it neighboring} communications. 
We also used the shorthand    $\prox_{\Lambda^k,R}(\X):=
    [\prox_{\alpha_1^k r_1}(x_1),$ $
    \cdots,
    \prox_{\alpha_m^k r_m}(x_m)]^{\top}.$
\begin{algorithm}[t!]
\centering
\resizebox{\columnwidth}{!}{%
  \begin{minipage}{\columnwidth}
\caption{Decentralized Adaptive Three Operator Splitting with local min-consensus (local\_DATOS)}
  \noindent \textbf{Data:} (i) Initialization:  $\alpha_i^{-1} \in (0, \infty)$, $\X^0,\Sub^0 \in \mathbb{R}^{m\times n}$  and  $\X^{-1},\AF^0,\D^0,\T^0 = 0$;  (ii) Backtracking   parameters: summable sequence $\{m^k\}$, $\eta\in(0,1)$, $\delta \in (0,1)$;  (iii) Gossip matrix  $W=(1-c)I+c\widetilde{W}$, $c\in(0,1/2)$.		
		\begin{algorithmic}[1]
           \State \texttt{(S.1) Communication Step: } $$\X^{k+1/2}=W\X^k,\quad \D^{k+1/2}=W(\nabla F(\X^k)+\Sub^k+\D^k);$$
          
			\State \texttt{(S.2) Decentralized line-search: }Each agent updates $\alpha_i^k$ according to:
   	 $$\widetilde{\alpha}_i^k=
   	     \sqrt{(\alpha_i^{k-1})^2+m^k},$$
    $$\alpha_i^k=\texttt{Linesearch}(\widetilde \alpha_i^{k},x_i^k,x_i^{k+1/2},-d_i^{k+1/2},\eta,\delta);$$
			\State \texttt{(S.3)} \texttt{Local min-consensus: }Each agent updates $\alpha_i^k$ according to:
$$\alpha_i^k\leftarrow\min_{j\in{\cal N}_i}\alpha_j^k;$$
 \State\texttt{(S.4) Extra scalar communication step: }Let $\Lambda^k:=\texttt{diag}(\alpha_1^k,\alpha_2^k,\cdots,\alpha^k_m)$,
            $$\D_{\Lambda}^k=(I-W)(\Lambda^k)^{-1}\X^k;$$
     \State \texttt{(S.5) Updates of the primal and dual variables:}
\begin{subequations}
\label{eq:localDATOS}
         \begin{align}
         \label{eq:localDATOS  A}
         \AF^{k+1}&=\X^{k+1/2}-\Lambda^k\D^{k+1/2},\\
         \label{eq:localDATOS X}
    \X^{k+1}&=\prox_{\Lambda^k, R}(\AF^{k+1}+\Lambda^k\Sub^k),\\
    \label{eq:localDATOS S}
    \Sub^{k+1}&=\Sub^k+(\Lambda^k)^{-1}(\AF^{k+1}-\X^{k+1}),\\
    \label{eq:localDATOS D}
\D^{k+1}&=\D^{k+1/2}+\D_{\Lambda}^k-\nabla F(\X^k)-\Sub^k;
     \end{align}
\end{subequations}
			\State \texttt{(S.6)}    If a termination criterion is not met,  $k\leftarrow k+1$ and go to step \texttt{(S.1)}.
\end{algorithmic}\label{alg:DATOS_local}
\end{minipage}
}
\end{algorithm}

In step \texttt{(S.2)}, the proposed stepsize updating  enforces \vspace{-0.2cm}
\begin{equation}
    \label{eq:adaptive stepsize selection local}
    0<\alpha_{\min}^k\leq \frac{\delta}{L^k},\quad \max_{i\in[m]}\left[(\alpha_i^k)^2-(\alpha_i^{k-1})^2\right]\leq m^k,\vspace{-0.1cm}
\end{equation}
where $\{m^k\}$ is a summable sequence and $\alpha^k_{\min}:= 
\min_{i\in[m]}\alpha_i^k$.  At a high level, using the \emph{same} slack sequence $\{m^k\}$ across agents prevents persistent
stepsize heterogeneity and is instrumental to ensure that eventually $\Lambda^k$ becomes close to
$\alpha_{\min}^k I$ (indeed equal, under a mild lower bound); this property is key in the convergence analysis.  
\begin{lemma}[finite-time min-consensus]
    \label{lemma:consensus stepsize}
Consider Algorithm~\ref{alg:DATOS_local} under Assumptions~\ref{ass:function} and~\ref{ass:W}.     If $\alpha_{\min}^k\geq \underline{\alpha}$, for some $\underline{\alpha}>0$ and   all  $k\geq 0$, then there exists $K'>1$ such that $\Lambda^k=\alpha^k_{\min} I$, for all $k\geq K'$.
\end{lemma}
\begin{proof} Define   $\alpha_{\max}^k:=\max_i \alpha_i^k$ and let $d_{\mathcal G}$ be the graph diameter.
    Since $\{m^k\}$ is summable, there exists $K>0$ such that \vspace{-0.2cm}
    \begin{equation}
    \label{eq:never come back 2}
        \sum_{t=k}^{\infty}m^t<\frac{1-\eta^2}{2}(\underline{\alpha})^2,\quad \forall k\geq K.\vspace{-0.1cm}
    \end{equation}
    Fix $k\ge K$. Suppose that at iteration $k$ some agent $i\in[m]$ \emph{decreases} its stepsize (i.e., it backtracks in
\texttt{(S.2)}).   Then necessarily $(\alpha_i^k)^2\leq \eta^2[(\alpha_i^{k-1})^2+m^k].$
Next, we relate this decrease at agent $i$ to $\alpha_{\max}$ after at most $d_{\mathcal G}$ iterations. By the local min-consensus step \texttt{(S.3)}, the value $\alpha_i^k$ is propagated through the network and affects all agents
within $d_{\mathcal G}$ iterations; meanwhile, by \eqref{eq:adaptive stepsize selection local},
the squared stepsizes can increase by at most $m^t$ at each iteration. Therefore, for any
$k'\ge k+d_{\mathcal G}$ it holds\vspace{-0.2cm}
\begin{equation*}
    \label{eq:max stepsize contraction}
    (\alpha_{\max}^{k'})^2\leq (\alpha_i^k)^2+\sum_{t=k+1}^{k'}m^t\overset{\eqref{eq:never come back 2}}{\leq}\left(\eta^2+\frac{1-\eta^2}{2}\right)(\alpha_i^{k-1})^2\leq \frac{\eta^2+1}{2}(\alpha_{\max}^{k-1})^2.\vspace{-0.1cm}
\end{equation*}
 Hence, $\alpha_i^k$ can decrease only finitely many times; otherwise, repeated contractions would force
$\alpha_{\max}^k\to 0$, contradicting the assumption $\alpha_{\min}^k\ge \underline{\alpha}>0$.
Therefore, there exists $K_0\ge K$ such that no agent decreases its stepsize for any $k\ge K_0$. Once no stepsizes decrease,   \texttt{(S.3)} makes all stepsizes equal
within at most $d_{\mathcal G}$ iterations. Therefore, there exists $K':=K_0+d_{\mathcal G}$ such that
$\Lambda^k=\alpha_{\min}^k I$ for all $k\ge K'$.
\end{proof}
\begin{remark}
Lemma~\ref{lemma:consensus stepsize} assumes $\alpha_{\min}^k$ is bounded away from $0$. This holds 
e.g., when each $f_i$ is {\it globally} smooth, or when $\{\X^k\}$ is bounded (for example, if   $r_i$'s contain the  indicator function of a compact set).
\end{remark}
  
With Lemma~\ref{lemma:consensus stepsize}, there exists $K'>1$ such that, for all $k\ge K'$, $\Lambda^k=\alpha_{\min}^k I$. Hence,  from iterate $K'$ onword,  
  Algorithm~\ref{alg:DATOS_local} enerates the same recursion as Algorithm~\ref{alg:DATOS}, with  the {\it uniform}  stepsize  
$\alpha^k:=\alpha_{\min}^k$. Therefore,  the convergence analysis of Algorithm~\ref{alg:DATOS} can be \emph{restarted at time $K'$}, taking
$({\scr X}^{K'},{\scr S}^{K'})$ as the new ``initial point''. Since the descent inequality in (\ref{eq:descent inequality sec 3}) holds for all   $k\ge K'$,   the asymptotic convergence of  Algorithm~\ref{alg:DATOS_local} is proved following the same arguments as in Theorem~\ref{thm:iterates convergence} (${\cal V}^{k}_{{\scr X}^*{\scr S}^*}$  is defined here as in (\ref{eq:Lyapunov_0}) but with $\alpha^k=\alpha_{\min}^k$). 

To establish a convergence rate for Algorithm~\ref{alg:DATOS_local}, we follow the  proof of  Theorem~\ref{thm:asymptotic convergence}, starting from   $k\geq K'$. The rate therein is derived via telescoping argument on~\eqref{eq:descent inequality sec 3} and it requires   a bounded set $\mathcal{B}_X\times\mathcal{B}_S$ containing the trajectory, so that the
partial gap $\mathcal{G}_{\mathcal{B}_X\times\mathcal{B}_S}$ is a valid measure of optimality.  Since $\{{\scr X}^k\}$ is already
controlled by the Lyapunov descent once the stepsize is uniform, it is sufficient to guarantee boundedness of the
dual iterates $\{{\scr S}^k\}$ after $K'$.     This is established next, using   summability of   $\{m^k\}$; ${\cal V}^{k}_{{\scr X}^*{\scr S}^*}$ below is defined as in (\ref{eq:Lyapunov_0})   with $\alpha^k=\alpha_{\min}^k$. 
\begin{lemma}
    \label{lemma:bounded error ball local}
      Under the conditions of Lemma~\ref{lemma:consensus stepsize}, for any $k\geq K'$, 
    \begin{equation}
        \label{eq:bounded error ball local}
        \|{\scr S}^{k}-{\scr S}^*\|^2\leq R_S:=(2{\cal V}^{K'}_{{\scr X}^*{\scr S}^*})/(\underline{\alpha})^2.
    \end{equation} 
\end{lemma}
\begin{proof}
    Without loss of generality, take $K'$ large enough so that \eqref{eq:never come back 2} holds with $k=K'$. We prove \eqref{eq:bounded error ball local} by induction on $k\geq K'$.   For $k=K'$, using $\alpha_{\min}^{K'}\ge \underline{\alpha}$,
    we have 
   $(\underline{\alpha})^2\|{\scr S}^{K'}-{\scr S}^*\|^2\leq {\cal V}^{K'}_{{\scr X}^*{\scr S}^*}$. 
    Assume now that $\|{\scr S}^{t}-{\scr S}^*\|^2\le R_S$ for all $t=K',\ldots,k$.
Since $\Lambda^t=\alpha_{\min}^t I$ for all $t\ge K'$ (Lemma~\ref{lemma:consensus stepsize}), starting from Lemma~\ref{lemma:descent lemma}, the same descent
estimate as in (\ref{eq:descent inequality sec 3}) (Lemma~\ref{lemma:legal stepsize}(\textbf{i})) can be proved  here from $K'$ onward, with the increment budget
controlled by $m^t$;   in particular, combining this result with the induction  hypothesis yields   
    $$(\underline{\alpha})^2\|{\scr S}^{k+1}-{\scr S}^*\|^2\leq {\cal V}^{K'}_{{\scr X}^*{\scr S}^*}+R_S\sum_{j=K'}^k m^j\overset{\eqref{eq:never come back 2}}{\leq} 2{\cal V}^{K'}_{{\scr X}^*{\scr S}^*}\Rightarrow \|{\scr S}^{k+1}-{\scr S}^*\|^2\leq R_S.$$
  This completed the proof by induction.
\end{proof}
Lemma~\ref{lemma:consensus stepsize} and  Lemma~\ref{lemma:bounded error ball local}   allow us obtain sublinear convergence of Algorithm~\ref{alg:DATOS_local} for all $k\geq K'$,   repeating the same  argument used in the proof of Theorem~\ref{thm:asymptotic convergence}. 
\begin{theorem}
    \label{thm:convergence local_DATOS}
Let $\{({\scr X}^k,{\scr A}^k,{\scr S}^k)\}$ be the   iterates generated by  Algorithm~\ref{alg:DATOS_local}, under  Assumptions~\ref{ass:function} and~\ref{ass:W}; and suppose    $\alpha_{\min}^k\geq \underline{\alpha}$, for some $\underline{\alpha}>0$ and   all  $k\geq 0$. Then $\{({\scr X}^k,{\scr S}^k)\}$ converges to some  $({\scr X}^*,{\scr S}^*)\in \mathcal P^\star\times \mathcal D^\star$.   Moreover, for any given $({\scr X}^*,{\scr S}^*)\in \mathcal P^\star\times \mathcal D^\star$, there exist  a bounded set   $\mathcal{B}_{X}\times\mathcal{B}_{S}$ and $K'\geq 1$, such that \smallskip 
 
\textbf{(i)} $({\scr X}^*,{\scr S}^*)\in \mathcal{B}_{X}\times\mathcal{B}_{S}$, and $({\scr A}^k,{\scr S}^k)\in \texttt{int}(\mathcal{B}_{X}\times\mathcal{B}_{S})$, for all $k\geq K'$; and \smallskip 

\textbf{(ii)} for any $k\geq K'$, it holds   \vspace{-0.2cm}
  \begin{equation}
        \label{eq:sublinea DATOS_II}
\mathcal{G}_{\mathcal{B}_x\times \mathcal{B}_S}(\overline{{\scr A}}^k,\overline{{\scr S}}^k)\leq\frac{1}{k-K'}\cdot\frac{C_{K'}+\left(1+\alpha^{K'-1}\sqrt{\frac{20}{(\underline{\alpha})^2}+3}\right)^2\mathcal{V}^{K'}_{{\scr X}^*,{\scr S}^*}}{ 2\cdot\min\left(\alpha^{-1},\delta/(2\widetilde{L}^k)\right)},\end{equation}
where $C_{K'}:=2\sum_{t=1}^{K'}\alpha^{t-1}_{\min}\mathcal{G}_{\mathcal{B}_X\times\mathcal{B}_S}({\scr A}^t,{\scr S}^t)$, and   $(\overline{{\scr A}}^k,\overline{{\scr S}}^k)$ is defined as in \eqref{eq:weighted_average} (with the convention that the averaging   starts at $K'$).
\end{theorem}
      Theorem~\ref{thm:convergence local_DATOS} certifies convergence of   Algorithm~\ref{alg:DATOS_local} at sublinar rate.  The primary distinction between the local-min and global-min consensus procedures is that the local-min does not guarantee a monotonically decrease of the merit function during the initial $K'$ iterations. Nevertheless, as shown numerically in Sec.~\ref{sec:numeric}, Algorithm~\ref{alg:DATOS_local} performs comparably to
Algorithm~\ref{alg:DATOS} while requiring only {\it neighbor} communications.

  
\section{Linear Convergence for Partly Smooth Functions}
\label{sec:linear} 
In this section,  we strengthen the convergence  guarantees of the proposed adaptive algorithms   by establishing
an asymptotic  \emph{linear rate}, for  instances of Problem~\eqref{eq:problem} where the agents’ losses are   (locally) strongly convex at the limit point 
and the nonsmooth (convex) term $R$ enjoys a \emph{partial smoothness} property. This is necessary because, for
Davis-Yin three-operator splitting, linear convergence generally \emph{cannot} be guaranteed under the sole
assumption that $F$ is strongly convex when $R$ is an arbitrary convex nonsmooth function, due to the possible
irregularity of subgradients (see, e.g., \cite[Theorem~3.5]{davis2017three}). 
Partly smooth   functions~\cite{lewis2002active} form a broad class that arises in many areas, including signal processing/machine learning, imaging, and control and system identification, among others \cite{liang2017activity}. 
Due to space limitations, we present the analysis only for Algorithm~\ref{alg:DATOS}; analogous conclusions can be established for Algorithm~\ref{alg:DATOS_local}.

\subsection{Partial Smoothness amd finite identification}
\label{subsec:partial smoothness}  
Let $\mathcal M$ be a $\mathcal C^2$ embedded submanifold of $\mathbb R^d$ around $x\in\mathcal M$. Given a function $r:\mathbb{R}^d\to\overline{\mathbb{R}}$, %
the restriction of $r$ to $\mathcal{M}$ is the function
$r|_{\mathcal{M}}:\mathcal{M}\to\overline{\mathbb{R}}$ defined by $r|_{\mathcal{M}}(z)=r(z)$
for all $z\in\mathcal{M}$. We say that $r|_{\mathcal{M}}$ is $\mathcal{C}^2$-smooth around $x\in\mathcal{M}$ if there
exist an open neighborhood $V\subseteq \mathbb{R}^d$ of $x$ and a $\mathcal{C}^2$ function
$\widehat r:V\to\mathbb{R}$ such that $\widehat r(z)=r(z)$ for all $z\in \mathcal{M}\cap V$;
any such $\widehat r$ is called a $\mathcal{C}^2$ representative (extension) of $r|_{\mathcal{M}}$
around $x$. 
\begin{definition}[Partly smoothness~{\cite[Def.~2.7]{lewis2002active}}]
\label{def:partial smoothness}
Given the set $\mathcal M\subseteq \mathbb R^d$, let $r:\mathbb{R}^d\to\overline{\mathbb{R}}$ be proper, lower semicontinuous, and convex,
and let $x\in \mathcal M$ be given such that $r$ has subgradient at $x$ and any point close to $x$ in $\mathcal M$. We say that $r$ is \emph{partly smooth at $x\in \mathcal M$ relative to}  $\mathcal{M}$
 if $\mathcal{M}$ is a $\mathcal{C}^2$-manifold around $x$ and    \begin{enumerate}
    \item [(i)]  \textbf{(Restricted smoothness)} $r|_{\mathcal{M}}$ is $\mathcal{C}^2$-smooth around $x$; 
    \item[(ii)]\textbf{Sharpness:}  $\TS_{\MN}(x) \equiv\texttt{par}(\partial r(x))^{\perp}$;
    \item[(iii)]\textbf{Continuity:} 
    $\partial r$ is inner semicontinuous at $x$ relative to $\MN$. 
\end{enumerate}
This class is denoted by 
{$\texttt{psf}_\MN(x)$}.
\end{definition}

By Theorem~\ref{thm:iterates convergence},  the iterates generated by Algorithm~\ref{alg:DATOS}
converge to an optimal   consensual point
\(\X^\star=\mathbf 1 (x^\star)^\top\),
for some solution \(x^\star\) of~\eqref{eq:problem}.
In the remainder of this section, we postulate the following \emph{local} structure at these limit points.
\begin{assumption}[Local structure around $x^\star$]\label{ass:local property} Let $\X^\star=\mathbf 1 (x^\star)^\top$ be the limit point of the sequence generated by Algorithm~\ref{alg:DATOS}, as provided by Theorem~\ref{thm:iterates convergence}.
Assume: 
\hfill
    \begin{itemize}
    \item[(i)]\textbf{(local strong convexity)}    There exist $\mu>0$ and a neighborhood $\mathcal U$ of $\X^\star$
such that $F$ is $\mu$-strongly convex on $\mathcal U$,   with $\mu>0$;
    \item[(ii)]\textbf{(partly smoothness)}    There exists a $\mathcal C^2$  manifold  $\mathcal M$ around $x^\star$
such that  $r_i\in  \texttt{psf}_{\MN}(x^\star)$, for every $i\in [m]$;   
        \item[(iii)] \textbf{Non-degeneracy:}    
      $-\nabla F(\X^\star)\in \texttt{ri}(\partial R(\X^\star))$. 
    \end{itemize}
\end{assumption}

Under Assumption\ref{ass:local property}(ii), by separability of partial smoothness ~\cite[Prop.~4.5]{lewis2002active},\vspace{-0.2cm}
\[
R(\X)=\sum_{i=1}^m r_i(x_i)\in \texttt{psf}_{\MN_R}(\X^\star),
\qquad
\MN_R \triangleq \mathcal M\times\cdots\times\mathcal M \subseteq \mathbb R^{m\times d}.\vspace{-0.2cm}
\]

 A key consequence of partial smoothness is the \emph{finite} time    identification: after a finite number
of iterations, the primal iterates of Algorithm~\ref{alg:DATOS} enter (and remain on) the active manifold containing the limit point. 
This property is formalized next.
\begin{theorem}[identification]
    \label{thm:finite activity identification}  Let $\{\X^k\}$ be the (primal) sequence generated by Algorithm~\ref{alg:DATOS} in the setting of Theorem~\ref{thm:iterates convergence},    
$\X^k\to \X^\star=\mathbf 1 (x^\star)^\top$.  If Assumption~\ref{ass:local property} holds at $\X^\star$, then there exists $K_2\ge K_1$
such that $\X^k\in \MN_R$, for all $k\ge K_2$.    
\end{theorem}
\begin{proof}By Theorem~\ref{thm:iterates convergence}, we have $\X^k\to \X^\star$ and $\Sub^k\to \Sub^\star$, with
$\Sub^k\in\partial R(\X^k)$ for all $k$, which together with the   outer semicontinuity of $\partial R$, yields  $\Sub^\star\in\partial R(\X^\star)$. Therefore, 
$\texttt{dist}(\Sub^*,\partial R(\X^k))\rightarrow 0$. 
Moreover, under Assumption~\ref{ass:function}(ii), each $r_i$ is subdifferentially continuous on its domain
(e.g., \cite[Ex.~13.30]{rockafellar2009variational}); 
hence, 
$R(\X^k)\to R(\X^\star)$.  
By Assumption~\ref{ass:local property}(ii), $R\in \texttt{psf}_{\MN_R}(\X^\star)$,  
and by Assumption~\ref{ass:local property}(iii) we have $\Sub^\star\in\texttt{ri}(\partial R(\X^\star))$.
Therefore, we can apply \cite[Thm.~5.3]{hare2004identifying} to the tilt function
\(
\X\mapsto R(\X)-\langle \Sub^\star,\X\rangle
\),
 and use 
$\X^k\to \X^\star$, $R(\X^k)\to R(\X^\star)$, and $\texttt{dist}(\Sub^\star,\partial R(\X^k))\to 0$ to conclude that there exists $K_2\geq K_1$ such that $\X^k\in \MN_R$, for all   $k\geq K_2$.
\end{proof}

Equipped with Theorem~\ref{thm:finite activity identification}, we can now establish a linear rate. Due to space limitations,
we present the argument for the case of  \emph{affine}  manifold $\MN_R$. This setting already covers
a broad range of structured nonsmooth functions of practical interest, including polyhedral gauges and piecewise linear models
(e.g., $\ell_1$, mixed $\ell_{1,2}$ and $\ell_\infty$ norms)

\subsection{Affine manifold $\MN_R$}
\label{subsec:affine}  Throughout this subsection, we assume  that $\MN_R$ in Assumption~\ref{ass:local property} is \emph{affine}. 
Since    $\X^\star\in\MN_R$ (Th.~\ref{thm:finite activity identification}),  we can write   $\MN_R=\X^\star+\mathcal T_{\MN_R}(\X^\star)$, where $\mathcal T_{\MN_R}(\X^\star)$ is a linear subspace of $\mathbb{R}^d$.  We   set 
$\mathcal T_R := \mathcal T_{\MN_R}(\X^\star)$ and  $\mathcal N_R := \mathcal N_{\MN_R}(\X^\star)=\mathcal T^\perp_R$; and we  denote the corresponding row-wise orthogonal projections by $P_{\mathcal T_R}(\cdot)$ and $P_{\mathcal N_R}(\cdot)$, respectively.  For     $\Z\in\mathbb R^{m\times d}$, we write $\Z=(\Z)_{\mathcal T_R}+(\Z)_{\mathcal N_R}$, with 
$(\Z)_{\mathcal T_R}:=P_{\mathcal T_R}(\Z)$ and $(\Z)_{\mathcal N_R}:=P_{\mathcal N_R}(\Z)$.

 The linear convergence analysis builds on the following  refined descent inequality that follows the proof of Lemma~\ref{lemma:descent lemma} now exploiting the \emph{local} strong convexity of $F$. 
 \begin{lemma}
\label{lem:decomposed_descent}
Under  Assumption~\ref{ass:local property}(i)-(ii) there exists $K_1\ge 0$ such that 
\begin{equation}
\label{eq:decomposed_descent}
\begin{aligned}
&\|\X^{k+1}-\X^\star\|^2
\\&\quad +(\alpha^k)^2\Big(\|(\Sub^{k+1})_{\mathcal T_R}-(\Sub^\star)_{\mathcal T_R}\|^2+\|(\Sub^{k+1})_{\mathcal N_R}-(\Sub^\star)_{\mathcal N_R}\|^2
+\|\Y^{k+1}-\Y^\star\|_{M^2}^2\Big)\\
&\leq  \|\X^{k}-\X^\star\|^2
\\&\quad +(\alpha^k)^2\Big(\|(\Sub^{k})_{\mathcal T_R}-(\Sub^\star)_{\mathcal T_R}\|^2+\|(\Sub^{k})_{\mathcal N_R}-(\Sub^\star)_{\mathcal N_R}\|^2
+\|\Y^{k}-\Y^\star\|_{M^2}^2\Big)\\
&-\Big(\alpha^k{\cal D}^k_{\X^\star}+(1-\delta)\|\AF^{k+1}-\X^k\|^2+(\alpha^k)^2\|\Y^k-\Y^{k+1}\|_{M^2}^2\Big),
\end{aligned}
\end{equation}
 for all $k\ge K_1$, where\vspace{-0.2cm}
\[
{\cal D}^k_{\X^\star}:=\frac{\alpha^k\mu}{2}\|\X^k-\X^\star\|^2
+\frac{\alpha^k}{2L_f}\|\nabla F(\X^k)-\nabla F(\X^\star)\|^2.
\]
\end{lemma}
\begin{proof}
      Repeating the derivation of Lemma~\ref{lemma:descent lemma} while using the local $\mu$-strong convexity
of $F$ around $\X^\star$ (Assumption~\ref{ass:local property}(i)), and the $L_f$-Lipschitz continuity of $\nabla F$ over   ${\cal H}$ (Lemma~\ref{lemma:legal stepsize}), we infer that there exists $K_1\ge 0$ such that, for all $k\ge K_1$,
\begin{equation}
\label{eq:tighter_descent}
 \mathcal V_{{\scr X}^\star {\scr S}^\star}^{k+1} 
\leq   \mathcal V_{{\scr X}^\star {\scr S}^\star}^{k}
-\|\widetilde{\AF}^{k+1}-\widetilde{\X}^k\|^2-(1-\delta)\|\AF^{k+1}-\X^k\|^2-\alpha^k{\cal D}_{\X^\star}^k.  
\end{equation}
The inequality~\eqref{lem:decomposed_descent} follow  from~\eqref{eq:tighter_descent} using $\widetilde{\AF}^{k+1}=\alpha^kM(\Y^k-\Y^{k+1})$ (due to~\eqref{eq:A}) and $\widetilde \X^k=\mathbf{0}$ (eq.~\eqref{eq:BCV X 1'}); and the orthogonal decomposition  of the subgradient variables, i.e.,    
$\|\Sub^{t}-\Sub^\star\|^2=\|(\Sub^{t})_{\mathcal T_R}-(\Sub^\star)_{\mathcal T_R}\|^2+\|(\Sub^{k})_{\mathcal N_R}-(\Sub^\star)_{\mathcal N_R}\|^2$, $t=k,k+1$. 
\end{proof}

We proceed  controlling the subgradient errors $\|(\Sub^{t})_{\mathcal T_R}-(\Sub^\star)_{\mathcal T_R}\|^2$ and $\|(\Sub^{k})_{\mathcal N_R}$ $-(\Sub^\star)_{\mathcal N_R}\|^2$ in \eqref{eq:decomposed_descent} leveraging the partial smoothness of $R$ (Assumption~\ref{ass:local property}).   
\begin{lemma} 
\label{lem:TN_control}
Assume Theorem~\ref{thm:finite activity identification} and that   $\MN_R$ therein 
is \emph{affine}.   Then,  
\begin{itemize}
\item[\bf (a)] {\bf Tangent component:} There exist $L_r\ge 0$ and $K_3\ge K_2$ such that  
\begin{equation}
\label{eq:tangent_holder_combined}
\|(\Sub^k)_{\mathcal T_R}-(\Sub^\star)_{\mathcal T_R}\|\le L_r\|\X^k-\X^\star\|,\quad \forall  k\ge K_3;
\end{equation} 
\item[\bf (b)] {\bf Normal component:} For all $k\ge K_2$,\begin{equation}
\label{eq:normal_holder_combined}
\|(\Sub^{k+1})_{\mathcal N_R}-(\Sub^\star)_{\mathcal N_R}\|
\le \|\nabla F(\X^k)-\nabla F(\X^\star)\|.
\end{equation}
\end{itemize}
\end{lemma}
\begin{proof}
\textbf{(a)}  Since   $R\in \texttt{psf}_{\MN_R}$, the restriction $R|_{\MN_R}$ is $\mathcal C^2$ around
$\X^\star$. Hence, there exist an open neighborhood $\mathcal U$ of $\X^\star$ and
$\widehat R\in\mathcal C^2(\mathcal U)$ such that $\widehat R(\X)=R(\X)$, for all $\X\in\MN_R\cap\mathcal U$. 
Denoting   $\nabla_{\MN_R}R(\X)$  the manifold gradient of $R|_{\MN_R}$ at $\X$ ~\cite[Def.~3.58]{boumal2023introduction}. Possibly shrinking $\mathcal U$ (still denoted by $\mathcal U$), the following holds:
$(\partial R(\X))_{\mathcal T_R}\overset{(a)}{=}\{\nabla_{\MN_R}R(\X)\}\overset{(b)}{=}\big(\nabla\widehat R(\X)\big)_{\mathcal T_R}$
for all  $\X\in\MN_R\cap\mathcal U$, 
where (a) and (b) follow from \cite[Fact 3.3]{liang2017activity} and  \cite[Prop.~3.61]{boumal2023introduction}, respectively.

 Since $\X^k\to\X^\star$ (Theorem~\ref{thm:iterates convergence}) and   $\X^k\in\MN_R$ for all $k\ge K_2$ (Theorem~\ref{thm:finite activity identification}),   
there exists $K_3\ge K_2$ such that $\X^k\in\MN_R\cap\mathcal U$, for all $k\ge K_3$. Thus, taking $\Sub^k\in\partial R(\X^k)$ and
$\Sub^\star\in\partial R(\X^\star)$,  yields $(\Sub^k)_{\mathcal T_R}-(\Sub^\star)_{\mathcal T_R}
=
\big(\nabla \widehat R(\X^k)\big)_{\mathcal T_R}-\big(\nabla \widehat R(\X^\star)\big)_{\mathcal T_R}$,  for all $k\geq K_3$; which together with    the non-expansiveness of  $P_{\mathcal T_R}$ and   smoothness of $\nabla \widehat R$  on  $\mathcal U$ (due to  $\widehat R\in\mathcal C^2(\mathcal U)$),   proves (\ref{eq:tangent_holder_combined}).  

\textbf{(b)}  Fix $k\ge K_2$. 
By Theorem~\ref{thm:finite activity identification}, $\X^k\in\MN_R$, 
hence
$\X^k-\X^\star\in\mathcal T_R$ and therefore 
$(\X^k)_{\mathcal N_R}=(\X^\star)_{\mathcal N_R}$ 
(recall that ${\mathcal N_R}$ is a linear subspace).  Noting  the commuting property  of row-wise projections with left multiplication, $(A\Z)_{\mathcal N_R} =A\,(\Z)_{\mathcal N_R}$ for all $\Z\in\mathbb R^{m\times d}$, we have  $W(\X^\star)_{\mathcal N_R}=(W\X^\star)_{\mathcal N_R}=(\X^\star)_{\mathcal N_R}$, where the last equality follows from 
$\X^\star=\mathbf 1(x^\star)^\top$  and $W1=1$; hence  $\X^\star_{\mathcal N_R}$ is also consensual. 

By Step~\texttt{(S.1)} of Algorithm~\ref{alg:DATOS} and the above relations, we have  \vspace{-.1cm}
\[
(\X^{k+1/2})_{\mathcal N_R}=(W\X^k)_{\mathcal N_R}=W(\X^k)_{\mathcal N_R}=W(\X^\star)_{\mathcal N_R}=(\X^\star)_{\mathcal N_R}.
\] Combining with  $(\X^{k+1})_{\mathcal N_R}=(\X^\star)_{\mathcal N_R}$  gives
\begin{equation}\label{eq:cancel_X_terms_cons}
(\X^{k+1/2}-\X^{k+1})_{\mathcal N_R}=0 \quad \text{and}
 \quad 
(\X^{k}-\X^{k+1/2})_{\mathcal N_R}=0. 
\end{equation}
From \eqref{eq:DATOS A}--\eqref{eq:DATOS S},\vspace{-.2cm}
\[
\Sub^{k+1}
=\Sub^k+\frac{1}{\alpha^k}\big(\X^{k+1/2}-\X^{k+1}\big)-\D^{k+1/2}.\vspace{-.1cm}
\]
Taking the $\mathcal N_R$-component and using \eqref{eq:cancel_X_terms_cons} yields
\begin{equation}\label{eq:SNR_intermediate_cons}
(\Sub^{k+1})_{\mathcal N_R}=(\Sub^k)_{\mathcal N_R}-(\D^{k+1/2})_{\mathcal N_R}.
\end{equation}
Denote $\nabla^k:=\nabla F(\X^k)$. By Step~\texttt{(S.1)}, $\D^{k+1/2}=W(\nabla^k+\Sub^k+\D^k)$; therefore
\begin{equation}\label{eq:DhalfNR_cons}
(\D^{k+1/2})_{\mathcal N_R}
=W\big((\nabla^k)_{\mathcal N_R}+(\Sub^k)_{\mathcal N_R}+(\D^k)_{\mathcal N_R}\big).
\end{equation}
Substituting \eqref{eq:DhalfNR_cons} into \eqref{eq:SNR_intermediate_cons} yields
\begin{equation}\label{eq:SNR_update_cons}
(\Sub^{k+1})_{\mathcal N_R}
=(\Sub^k)_{\mathcal N_R}-W\big((\D^k)_{\mathcal N_R}+(\Sub^k)_{\mathcal N_R}+(\nabla^k)_{\mathcal N_R}\big). 
\end{equation}

We now turn \eqref{eq:SNR_update_cons} into a bound involving \emph{only} the  gradient differences $(\nabla^k)_{\mathcal N_R}-(\nabla^{k-1})_{\mathcal N_R}$.  To do so, we eliminate the term $(\D^k)_{\mathcal N_R}$. From \eqref{eq:DATOS D} and \eqref{eq:cancel_X_terms_cons},  
$(\D^{k+1})_{\mathcal N_R}
=(\D^{k+1/2})_{\mathcal N_R}-(\nabla^k)_{\mathcal N_R}-(\Sub^k)_{\mathcal N_R}$.
Substituting \eqref{eq:DhalfNR_cons} into  this expression   yields\vspace{-0.2cm}
\begin{equation}\label{eq:DNR_update_cons_flow}
(\D^{k+1})_{\mathcal N_R}
=
W\Big((\D^k)_{\mathcal N_R}+(\Sub^k)_{\mathcal N_R}+(\nabla^k)_{\mathcal N_R}\Big)
-(\Sub^k)_{\mathcal N_R}-(\nabla^k)_{\mathcal N_R}.
\end{equation}
Adding \eqref{eq:SNR_update_cons} and \eqref{eq:DNR_update_cons_flow} cancels the $W(\cdot)$ terms and gives the
 {explicit} relation
$(\Sub^{k+1})_{\mathcal N_R}$ $+(\D^{k+1})_{\mathcal N_R}=-\,(\nabla^k)_{\mathcal N_R}$,
which substituted  
   in \eqref{eq:SNR_update_cons}  yields\vspace{-0.1cm}
\begin{equation}\label{eq:SNR_telescope_step_flow}
(\Sub^{k+1})_{\mathcal N_R} - (\Sub^k)_{\mathcal N_R}=-W\Big((\nabla)^k_{\mathcal N_R}-(\nabla^{k-1})_{\mathcal N_R}\Big),
\qquad \forall k\ge K_2+1.\vspace{-0.1cm}
\end{equation}

Fix $k\ge K_2$ and let $N\ge k+1$. Summing \eqref{eq:SNR_telescope_step_flow} over $t=k+1,k+2,\ldots,N$   yields\vspace{-0.1cm}
\begin{equation*}
(\Sub^{N+1})_{\mathcal N_R}-(\Sub^{k+1})_{\mathcal N_R}
=
-\,W\big((\nabla^{N})_{\mathcal N_R}-(\nabla^{k})_{\mathcal N_R}\big),\quad  k\geq K_2\,\,\text{and}\,\, N\geq k+1.\vspace{-0.1cm}
\end{equation*}
Letting $N\to\infty$   and using
$\Sub^{N+1}_{\mathcal N_R}\to \Sub^\star_{\mathcal N_R}$ and $(\nabla^{N})_{\mathcal N_R}\to (\nabla F(\X^\star))_{\mathcal N_R}$, we obtain
 \begin{equation}\label{eq:SNR_clean_identity_flow}
(\Sub^{k+1})_{\mathcal N_R}-(\Sub^\star)_{\mathcal N_R}
=
W\big((\nabla F(\X^\star))_{\mathcal N_R}-(\nabla F(\X^k))_{\mathcal N_R}\big),
\qquad \forall k\ge K_2.
\end{equation}
The inequality  \eqref{eq:normal_holder_combined} follows readily from  (\ref{eq:SNR_clean_identity_flow}) using $\|W\|_2\le 1$ and the nonexpansiveness of the orthogonal projector. 
\end{proof}

We proceed now proving linear convergence. Let $K_3$ be as in Lemma~\ref{lem:TN_control}.
 For $k\ge K_3$, define the Lyapunov function
\begin{equation}\label{eq:V_affine_final}
\begin{aligned}
\mathcal V^{k+1}_{\rm aff}
:=\;&\|\X^{k+1}-\X^\star\|^2
+(\alpha^k)^2\|(\Sub^{k+1})_{\mathcal T_R}-(\Sub^\star)_{\mathcal T_R}\|^2\\
&+\Big(1+\tfrac{1}{4\alpha^kL_f}\Big)(\alpha^k)^2
\|(\Sub^{k+1})_{\mathcal N_R}-(\Sub^\star)_{\mathcal N_R}\|^2
+(\alpha^k)^2\|\Y^{k+1}-\Y^\star\|_{M^2}^2 .
\end{aligned}
\end{equation}
Moreover, denote\vspace{-0.2cm}
\begin{equation}\label{eq:gamma_def_final}
\gamma^k:=\frac{(\alpha^k)^2}{(\alpha^{k-1})^2},\qquad k\ge 0,\vspace{-0.2cm}
\end{equation}
and define, for $k\ge K_3$, 
\begin{equation}\label{eq:rho_def_final}
\!\!\!\rho_{\rm aff}^k:=\!
\begin{cases}
\dfrac{\max\Big\{\frac{1}{\alpha^k}\|\X^k-\X^\star\|,\;\|\nabla F(\X^k)-\nabla F(\X^\star)\|,\;\|\Y^{k+1}-\Y^k\|_{M^2}\Big\}}
{\|\Y^k-\Y^\star\|_{M^2}}, & \!\!\Y^k\neq \Y^\star;\\[0.9em]
+\infty,& \!\!\Y^k=\Y^\star .
\end{cases}
\end{equation}
Using Lemma~\ref{lem:TN_control} in  \eqref{eq:decomposed_descent}, 
we obtain the following descent result for $\mathcal V^{k}_{\rm aff}$.  
 \begin{proposition} \label{lem:one_step_quasi_contr}
Assume the  setting of Lemma~\ref{lem:decomposed_descent} \& Lemma~\ref{lem:TN_control}, and  that  $\MN_R$ is affine.  Then,  holds that   
$\mathcal V^{k+1}_{\rm aff}
\leq \vartheta^k \,\,\mathcal V^{k}_{\rm aff},$ for all $k\ge K_3$,
where 
\begin{align}\label{eq:vartheta_def_final}
   \vartheta^k:=&\max\Bigg\{
1-\frac{\alpha^k\mu}{8},\;
\gamma^k\Big(1-\frac{\alpha^k\mu}{8L_r^2(\alpha^{k-1})^2}\Big),\;
\frac{\gamma^k}{1+\frac{1}{4\alpha^{k-1}L_f}},\; \nonumber\\
&\qquad\quad   
1-(\rho_{\rm aff}^k)^2 \min\Big\{\frac{\alpha^k\mu}{4},\;\frac{1}{4\alpha^kL_f},\;1\Big\}
\Bigg\}.
\end{align}
\end{proposition}
The above result ensures a sufficient decrease of $\mathcal V^k_{\rm aff}$ when $\rho_{\rm aff}^k$ is not too small.
The next lemma shows that if $\rho_{\rm aff}^k$ is small, then $\rho_{\rm aff}^{k+1}$ must be large enough. 

\begin{lemma}
    \label{lemma:rho^k+1}
  Assume the setting of Proposition~\ref{lem:one_step_quasi_contr}.
  \begin{itemize}
       \item[\textbf{(a)}]  If $\rho_{\texttt{aff}}^k<1$ and $k\geq K_3$, then \vspace{-0.2cm}
    \begin{equation}
        \label{eq:rho^k+1}
       \rho_{\texttt{aff}}^{k+1} \geq \frac{1}{\sqrt{\gamma^{k+1}}}\cdot \frac{1}{1+\alpha^kL_r}\left(\sqrt{c(1-\lambda_2(\widetilde{W}))}-\frac{3\rho_{\texttt{aff}}^k}{(1-\rho_{\texttt{aff}}^k)}\right); \end{equation}
       \item[\textbf{(b)}]  For every $k\ge K_3$,  \vspace{-0.3cm}   \begin{equation}
         \label{eq:rho_two_step_lb}   \max\{\rho_{\texttt aff}^k,\rho_{\texttt aff}^{k+1}\}
    \ge\
\underline{\rho}^k_{\texttt aff}:=\frac{\sqrt{c(1-\lambda_2(\widetilde W))}}{(1+\alpha^k L_r)\sqrt{\gamma^{k+1}}+\sqrt{c(1-\lambda_2(\widetilde W))}}.
       \end{equation}
  \end{itemize}
\end{lemma}

 \textit{Proof.}   We begin proving \eqref{eq:rho^k+1}.  
Using the   updates  of Algorithm~\ref{alg:DATOS},  
\begin{equation}
    \label{eq:X^k+1}
    \X^{k+1}=\X^k-\alpha^k(\Sub^{k+1}+\D^{k+1}+\nabla F(\X^k)).
\end{equation} 
By Theorem~\ref{thm:iterates convergence} and $\alpha^k\geq\underline{\alpha}>0$,   it follows  
$\Sub^\star+\D^\star+\nabla F(\X^\star)=0$. Subtracting $\X^\star$ on the LHS and $\alpha^k(\Sub^\star+\D^\star+\nabla F(\X^\star))=0$ on the RHS, we obtain: for $k\geq K_3$, 
\begin{equation}
\label{eq:k+1 decent init}
    \begin{aligned}    & \frac{1}{\alpha^k}\|\X^{k+1}-\X^*\| \\&=\left\|(\D^{k+1}-\D^\star)-  \left[\frac{1}{\alpha^k}(\X^{k}-\X^*)-(\Sub^{k+1}-\Sub^\star)-(\nabla F(\X^k)-\nabla F(\X^*))\right]\right\|    \\
   &  \geq 
    \|\PL(\Y^{k+1}-\Y^*)\|-2\|\nabla F(\X^k)-\nabla F(\X^*)\|-\frac{1}{\alpha^k}\|\X^k-\X^*\|-L_r\|\X^{k+1}-\X^*\|, 
    \end{aligned}
\end{equation}
where  the inequality follows from decomposing   $\Sub^{k+1}-\Sub^\star$ into $\mathcal T_R/\mathcal N_R$ components and applying  Lemma~\ref{lem:TN_control}  to bound  $\|(\Sub^{k+1})_{\mathcal T_R}-(\Sub^\star)_{\mathcal T_R}\|$  and $\|(\Sub^{k+1})_{\NS_R}-\Sub^\star)_{\NS_R}\|$.    

Next,  in (\ref{eq:k+1 decent init}), we use the lower bound  $\|\PL(\Y^{k+1}-\Y^\star)\| \geq\sqrt{c(1-\lambda_2(\widetilde W))}\,\|\Y^{k+1}$ $-\Y^\star\|$ (note that $\mathbf{0}\neq\Y^{k+1}-\Y^\star\in\texttt{span}(\PL)$)   and the definition of $\rho_{\texttt{aff}}^k$ (eq.~\eqref{eq:rho_def_final}); rearranging terms we obtain: for $k\geq K_3$,
\begin{equation}
\label{eq:k+1 decent}
         \frac{1}{\alpha^k}\|\X^{k+1}-\X^*\| 
        \geq  \frac{1}{1+\alpha^kL_r}\left(\sqrt{c(1-\lambda_2(\widetilde{W}))}\|\Y^{k+1}-\Y^*\|-3\rho_{\texttt{aff}}^k\|\Y^k-\Y^*\|_{M^2}\right).
\end{equation}
Moreover, still by~\eqref{eq:rho_def_final} and $M^2=W$, we have
\begin{equation}
    \label{eq:Y^k+1}
    \|\Y^{k+1}-\Y^*\|\geq \|\Y^k-\Y^*\|_{M^2}-\|\Y^{k+1}-\Y^k\|_{M^2}\geq (1-\rho_{\texttt{aff}}^k)\|\Y^k-\Y^*\|_{M^2}.
\end{equation}

Using~\eqref{eq:k+1 decent}, \eqref{eq:Y^k+1}, and  $\rho_{\texttt{aff}}^k<1$, we finally obtain:  for $k\geq K_3$, 
 $$ \rho_{\texttt{aff}}^{k+1} \geq     
  { \frac{\alpha^k}{\alpha^{k+1}}} \frac{1}{\alpha^k}\frac{\|\X^{k+1}-\X^*\|}{\|\Y^{k+1}-\Y^*\|}\geq \frac{1/\sqrt{   \gamma^{k+1}}}{1+\alpha^kL_r}\left(\sqrt{c(1-\lambda_2(\widetilde{W}))}-\frac{3\rho_{\texttt{aff}}^k}{(1-\rho_{\texttt{aff}}^k)}\right).$$
  This proves~\eqref{eq:rho^k+1}. 
  
  We prove now \eqref{eq:rho_two_step_lb}.  First consider the case $\rho_{\rm aff}^k\ge 1$. Since $\underline\rho_{\rm aff}^k<1$,
we have
$
\max\{\rho_{\rm aff}^k,\rho_{\rm aff}^{k+1}\}\ \ge\ 1\ \ge\ \underline\rho_{\rm aff}^k.
$ Now consider the case $\rho_{\rm aff}^k<1$. Set $x:=\rho_{\rm aff}^k\in[0,1)$,  and define
\[
a:=\frac{\sqrt{c(1-\lambda_2(\widetilde W))}}{(1+\alpha^k L_r)\sqrt{\gamma^{k+1}}},
\qquad
b:=\frac{3}{(1+\alpha^k L_r)\sqrt{\gamma^{k+1}}}.
\]
Then \eqref{eq:rho^k+1} reads\vspace{-0.1cm}
\[
\rho_{\rm aff}^{k+1}\ \ge\ a-b\,\frac{x}{1-x};\quad \text{and}\quad \max\{\rho_{\rm aff}^k,\rho_{\rm aff}^{k+1}\}
\ \ge\
\max\left\{x,\;a-b\,\frac{x}{1-x}\right\}.
\]
It is not difficult to prove \vspace{-0.1cm}
 \[
\max\left\{x,\;a-b\,\frac{x}{1-x}\right\}\ \ge\  \frac{a}{a+b+1},\qquad \forall x\in[0,1).
\] Substituting $a$ and $b$ in the lower bound yields \eqref{eq:rho_two_step_lb}.
  \hfill $\square$

Using  Lemma~\ref{lemma:rho^k+1} in Proposition~\ref{lem:one_step_quasi_contr},   leads to the following contraction property of  $ {\cal V}^{k+2}_{\texttt{aff}}$  every two consecutive iterations.  
\begin{theorem}[Linear convergence] 
    \label{thm:contraction} Let $\{\X^k,\Sub^k,\Y^k\}$ be generated by Algorithm~\ref{alg:DATOS}   in the setting of Theorem~\ref{thm:iterates convergence},    
$\X^k\to \X^\star=\mathbf 1 (x^\star)^\top$. Further assume  Assumption~\ref{ass:local property} at $\X^\star$, with $\MN_R$ affine.  Then
    \begin{equation}
        \label{eq:contraction}
         {\cal V}^{k+2}_{\texttt{aff}}\leq  \max\{\gamma^k\gamma^{k+1},1\}\cdot (1-r^k) \cdot {\cal V}^k_{\texttt{aff}},\quad \forall k\geq K_3, 
        \end{equation}
          where   $K_3\geq 1$ is such that Lemma~\ref{lem:TN_control} holds for all $k\ge K_3$, and  \vspace{-.1cm}\begin{equation*}
              r^k:=\min\left\{\frac{\alpha^k\mu}{8},\frac{\alpha^k\mu}{8L_r^2(\alpha^{k-1})^2},\frac{1}{1+4\alpha^{k-1}L_f}, (\underline{\rho}^k_{\texttt aff})^2\cdot\min\left\{\frac{\alpha^k\mu}{4},\frac{1}{4\alpha^kL_f},1\right\}\right\}. 
          \end{equation*}
         Furthermore,   there exists an index $K_4\ge K_3$ such that
\begin{equation}
\label{eq:q_const_def_thm}
r_{\rm aff}
:=\sup_{k\ge K_4}\Big(\max\{\gamma^k\gamma^{k+1},\,1\}\cdot(1-r^k)\Big)\;\in\;(0,1),
\end{equation}
and consequently\vspace{-.3cm}
\begin{equation}
\label{eq:R_linear_V_thm}
\mathcal V_{\rm aff}^{k}\;\le\; r_{\rm aff}^{\lfloor (k-K_4)/2\rfloor}\,\mathcal V_{\rm aff}^{K_4},
\qquad \forall k\ge K_4.
\end{equation}
Therefore  $\|\X^k-\X^\star\|\to 0$   $R$-linearly.
\end{theorem}
\begin{proof} The inequality  \eqref{eq:contraction} follow readily  applying  twice $\mathcal V^{k+1}_{\rm aff}
\leq \vartheta^k \,\,\mathcal V^{k}_{\rm aff}$ (Proposition~\ref{lem:one_step_quasi_contr})--yielding   $\mathcal V^{k+2}_{\rm aff}\ \le\ \vartheta^{k+1}\vartheta^{k}\,\mathcal V^{k}_{\rm aff}$--
factoring out
$\max\{\gamma^k\gamma^{k+1},$ $1\}$ from $\vartheta^{k+1}\vartheta^{k}$, and finally using the lower bound (\ref{eq:rho_two_step_lb}) (Lemma~\ref{lemma:rho^k+1}). 

We are left to prove that $r_{\texttt{aff}}\in (0,1)$. \textbf{(i)} Using $0<\underline\alpha\le \alpha^k\le \overline\alpha$ and $\gamma^k=\frac{(\alpha^k)^2}{(\alpha^{k-1})^2}\le \frac{\overline\alpha^2}{\underline\alpha^2}<+\infty$ for all $k$,  we  deduce    $r^k\geq \underline{r}>0$ for all $k\geq K_3$ and some $\underline r$ (each term in the minimum defining $r^k$ is bounded below by a strictly positive constant). \textbf{(ii)} Furthermore,   since $\alpha^k\to \alpha^\infty>0$, we have
$\gamma^k\to \frac{(\alpha^\infty)^2}{(\alpha^\infty)^2}=1$, implying $\max\{\gamma^k\gamma^{k+1},1\} \to 1.$  Pick any $\varepsilon\in\bigl(0, {\underline r}/({2(1-\underline r))}\bigr)$.
By the above limit, there exists $K_4\ge K_3$ such that, for all $k\ge K_4$, $\max\{\gamma^k\gamma^{k+1},1\}\ \le\ 1+\varepsilon.$ Combining \textbf{(i)} and \textbf{(ii)}, we deduce \vspace{-0.2cm}\[
r_{\rm aff}^k =\max\{\gamma^k\gamma^{k+1},1\}(1-r^k)
\ \le\ (1+\varepsilon)(1-\underline r)
\ <\ 1-\frac{\underline r}{2}
=: r_{\rm aff}\ \in(0,1).
\] This proves \eqref{eq:q_const_def_thm}.
\end{proof}
 From Theorem~\ref{thm:contraction} and following the same line of proof as in \cite[Lemma 12]{kuruzov2024achieving} we can finally obtain the following complexity for an $\varepsilon$ solution.  
\begin{corollary}
\label{cor:N epsilon}
   Assume the setting of Theorem~\ref{thm:contraction}; and let Let $\kappa_f:=L_f/\mu$ and $\kappa_r:=L_r/\mu$, Then $\|\X^k-\X^\star\|\leq \varepsilon$, for all $k\geq N_\varepsilon$, with  \vspace{-0.1cm}
    \begin{equation}
        \label{eq:N epsilon}
N_{\varepsilon}=\mathcal{O}\left(\frac{\kappa_f(1+\kappa_r^2)}{1-\lambda_2(\widetilde{W})}\log\left(\frac{1}{\varepsilon}\right)\right).
    \end{equation}
\end{corollary}

Corollary~\ref{cor:N epsilon} shows that the proposed algorithms enter a \emph{linear} convergence regime \emph{after} the finite-time manifold identification in Theorem~\ref{thm:finite activity identification}. In \eqref{eq:N epsilon}, the network dependence is standard via the inverse spectral gap, while partial smoothness affects the rate only through $(1+\kappa_r^2)$, where $\kappa_r=L_r/\mu$ and $L_r$ comes from the tangent subgradient control in Lemma~\ref{lem:TN_control}(a). Thus, the complexity matches existing non-adaptive methods up to constants, with an extra factor capturing the local variability of $R$ along the active manifold. 
This factor vanishes for polyhedral penalties (e.g., $r_i(x)=\lambda_i\|x\|_1$), for which $L_r=0$ (hence $\kappa_r=0$) after identification. In this regime, DATOS recovers exactly the same asymptotic iteration complexity as the aforementioned non-adaptive baselines, while retaining the benefit of an adaptive stepsize selection prior to (and after) identification.

\section{Numerical Results}
\label{sec:numeric}
This section reports numerical experiments comparing Algorithms~\ref{alg:DATOS} (\texttt{global\_DATOS})  and~\ref{alg:DATOS_local} (\texttt{local\_DATOS}) with representative decentralized baselines. As nonadaptive benchmarks, we consider SONATA~\cite{sun2019distributed} and PG-EXTRA~\cite{shi2015proximal}, both relying on constant stepsizes that depend on (global) problem and network parameters. For these methods, we perform a manual grid-search over admissible stepsizes and report the best-performing choice yielding stable and fast convergence. We further include decentralized implementation  of the adaptive \emph{centralized} method adaPDM~\cite{Latafat_23b}. We stress that adaPDM is not designed to be decentralized; when implemented  in decentralized settings, it is not fully parameter-free, as it requires the knowledge at the agents' side of the global network-related quantity $\|I-W\|$.  To illustrate the effect of such knowledge, we simulate two variants: {\bf (i)} \texttt{adaPDM}--assuming access to   $\|I-W\|(\leq 2)$; and {\bf (ii)} \texttt{adaPDM2}--using the conservative network-agnostic upper bound $2$ for  $\|I-W\|2$.   In both cases, we manually tune the parameter $t$ (as appears in~\cite{Latafat_23b}) to obtain the best observed convergence behavior.  We anticipate that both adaPDM and adaPDM2 are very sensitive to the choice of $t$.

   We generate Erd\H{o}s--R\'enyi graphs with $m=20$ agents and edge probability $p\in\{0.1,0.5,0.9\}$, spanning poorly to well connected regimes. All methods employ Metropolis-Hastings gossip weights (see, e.g.,~\cite{Nedic_Olshevsky_Rabbat2018}). For Algorithms~\ref{alg:DATOS} and~\ref{alg:DATOS_local}, we initialize $\X^0$ and $\Sub^0$ randomly and set $\alpha^{-1}=10$, $\delta=0.9$, and $c=1/3$.

\subsection{Logistic regression with $\ell_1$-regularization}
Consider  the decentralized logistic regression problem with $\ell_1$-regularization, which is an instance of \eqref{eq:problem}, with \vspace{-.2cm}
\begin{equation*}\label{eq:logistic-scvx}f_i(x)=\frac{1}{n}\sum_{j=1}^{n}\log(1+\exp(-b_{ij}\cdot\langle x,a_{ij}\rangle)),\quad r(x)=\lambda\|x\|_1,\vspace{-.2cm}\end{equation*}
where  $a_{ij}\in \mathbb{R}^{d}$ and $b_{ij}\in\{-1,1\}$. The data set  $\{(a_{ij},b_{ij})\}_{j=1}^n$ is  owned   by agent $i$. We use the MNIST dataset from LIBSVM~\cite{chih2libsvm}, taking  the first $N=6000$   samples   (hence $n=300$). The feature dimension is $d=784$. We set  $\lambda=10^{-5}$. 

Figure~\ref{fig:general cvx}   plots the optimality gap  $({1}/{m})\sum_{i=1}^mu(x_i^\nu)-u^*$ achieved by all the algorithms versus the number of iterations $\nu$,       where   $u^*$ is the minimum value of $u$.   
Both \texttt{global\_DATOS} and \texttt{local\_DATOS} consistently outperform SONATA and PG-EXTRA,  while requiring no manual stepsize tuning. Moreover, they improve markedly over adaPDM in both variants (\texttt{adaPDM} and \texttt{adaPDM2}).  Finally, the performance gap between \texttt{local\_DATOS} (local min-consensus) and \texttt{global\_DATOS} (global min-consensus) is negligible, particularly on well connected graphs.

\begin{figure*}[htbp] 
    \centering
    \includegraphics[width=0.32\linewidth]{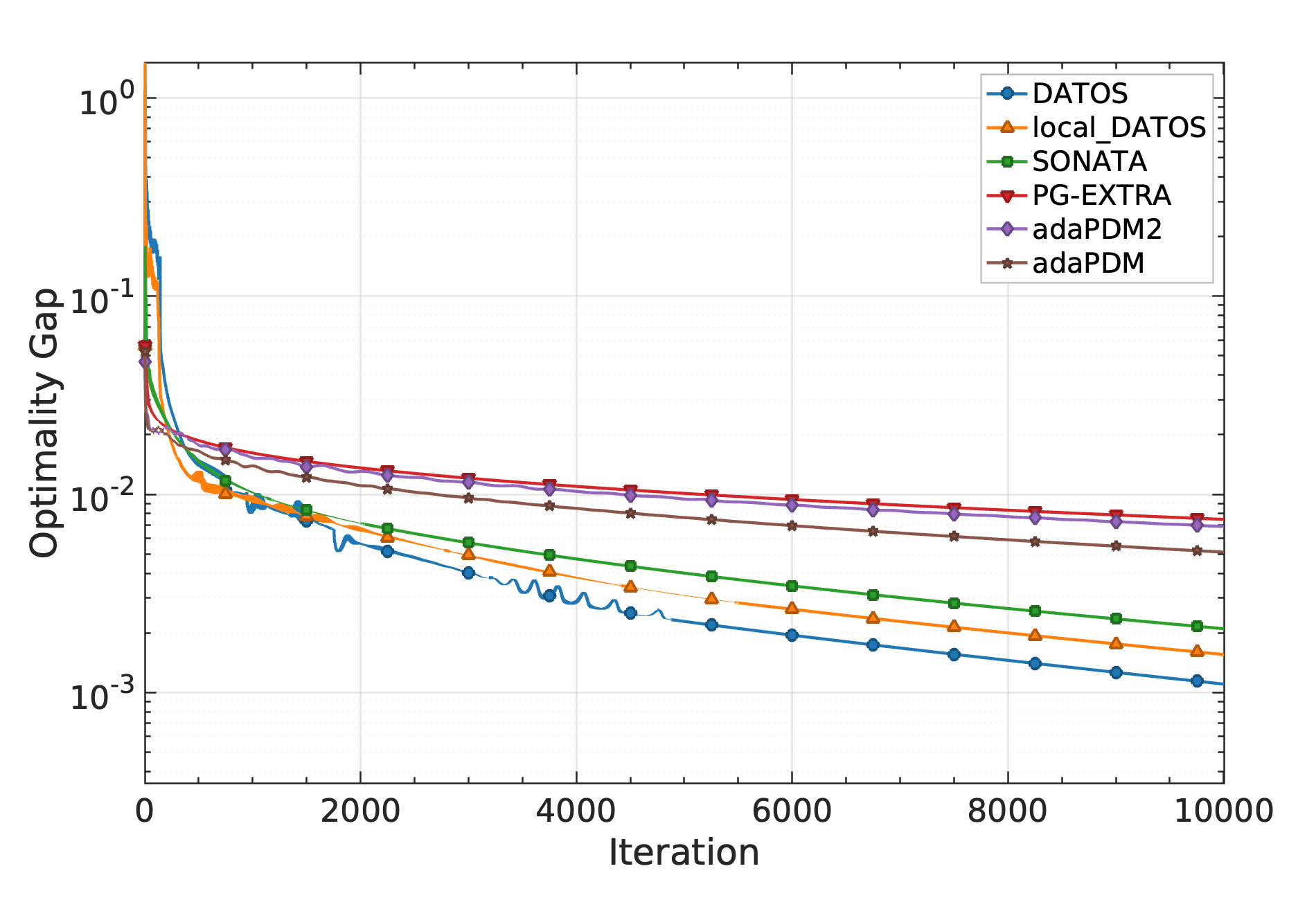}
\includegraphics[width=0.32\linewidth]{siam_sim/siam_logistic_1.png}
\includegraphics[width=0.32\linewidth]{siam_sim/siam_logistic_1.png}
\caption{Logistic regression with $\ell_1$-regularization: ${\frac{1}{m}\sum_{i=1}^m u(x_i)-u(x^*)}$ v.s. \# iterations.  Comparison of PG-EXTRA, SONATA, adaPDM, adaPDM2, global\_DATOS and local\_DATOS on Erdos-Renyi graphs with   edge-probability   $p=0.1$  (left);   $p=0.5$ (middle); and  $p=0.9$ (right).}\vspace{-0.3cm} 
\label{fig:general cvx}
\end{figure*}

\subsection{Maximum Likelihood (ML) estimate of the covariance matrix}
We consider decentralized estimation of an inverse covariance matrix, an instance of~\eqref{eq:problem} with variable $X\in\mathbb{S}^d_{++}$ and\vspace{-0.2cm}
\begin{equation*}\label{eq:maximum likelihood}f_i(X)=-n(\log(\texttt{det}(X)))-\texttt{trace}(XY_i),\qquad r(X)=\delta_C(X),\end{equation*}
 where $C=\{X\in\mathbb{S}^{d}_{++}: a I\preceq X\preceq bI\}$, for some $0<a\leq b$, and $Y_i=\frac{1}{n}\sum_{j=1}^n y_j^i (y_j^i)^{\top}$ for $\{y_j^i\}$ sampled from   a Gaussian distribution with covariance matrix $\Sigma\in\mathbb{S}^d_{++}$. We set $n=100$ and $d=5$. This is a representative case where $f$ is only \emph{locally} smooth; consequently, SONATA and PG-EXTRA do not come with global convergence guarantees, and stable behavior requires conservative stepsize choices. We again tune their stepsizes for best stable performance.

Figure~\ref{fig:general local smooth} reports $\bar u^k-u^\star$ versus the iteration counter, where $u^\star$ is computed via a centralized proximal-gradient method with line-search up to tolerance $10^{-30}$. The plots show that \texttt{global\_DATOS} and \texttt{local\_DATOS} consistently outperform  SONATA and PG-EXTRA--which require conservative stepsize selections to ensure stable global convergence--including both adaPDM variants. We also note that, in this experiment, the iterates generated by  \texttt{local\_DATOS} (and \texttt{global\_DATOS}) remain bounded due to the compact constraint set enforced by $r(X)=\delta_C(X)$. Therefore, the boundedness /local-smoothness hypotheses required by Theorem~\ref{thm:convergence local_DATOS} are automatically satisfied.

\begin{figure*}[htbp] 
    \centering
    \includegraphics[width=0.32\linewidth]{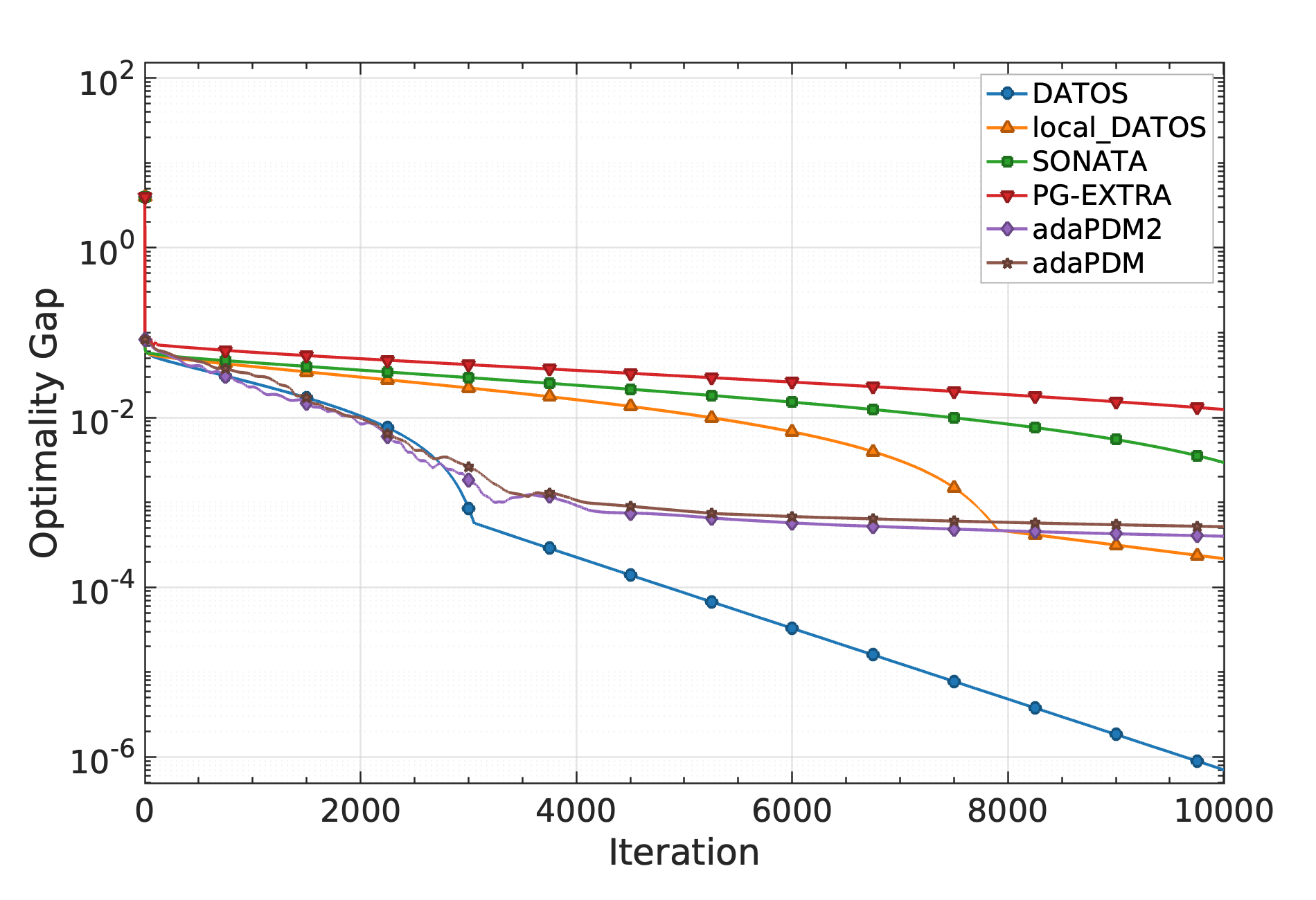}  \includegraphics[width=0.32\linewidth]{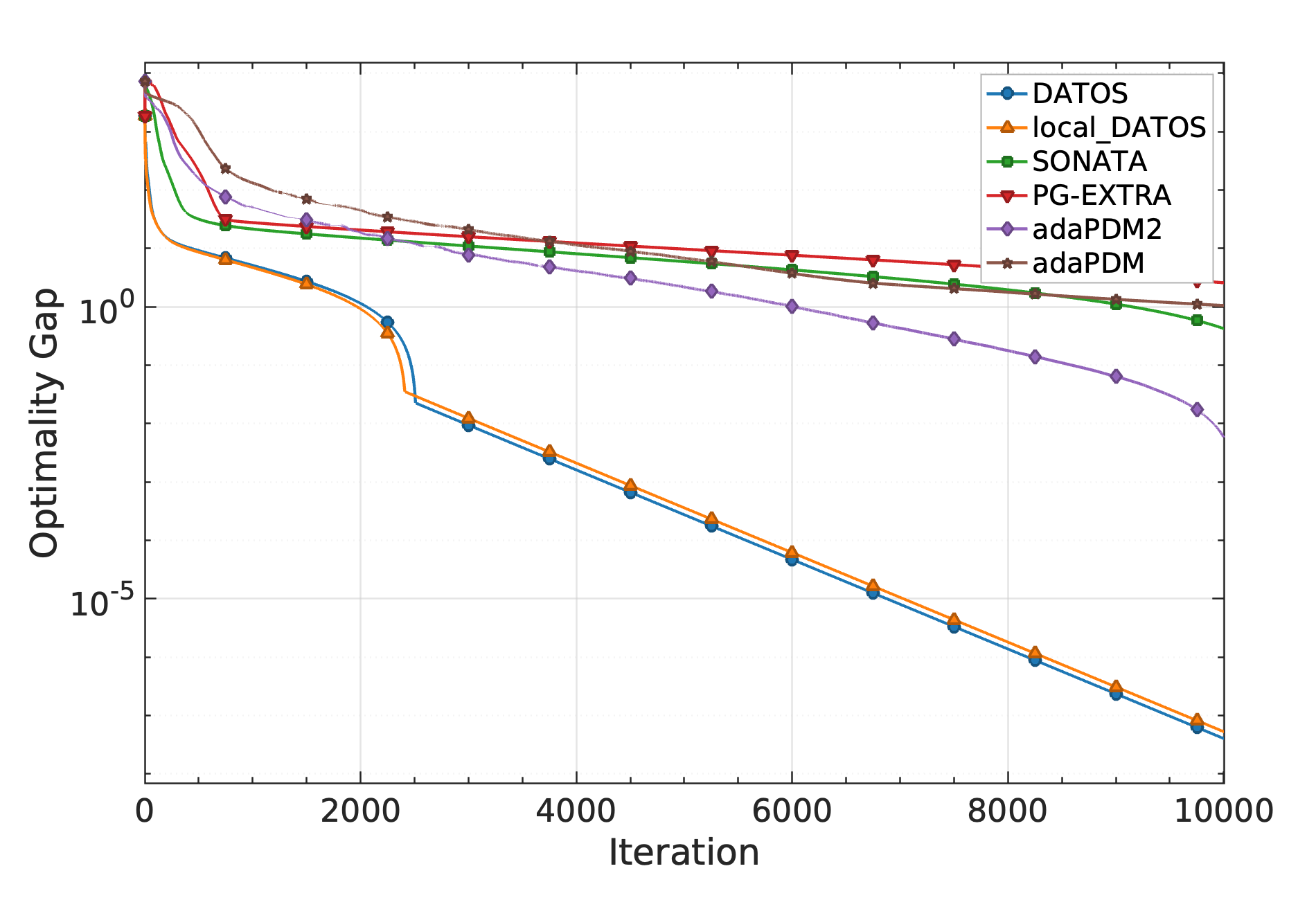}
\includegraphics[width=0.32\linewidth]{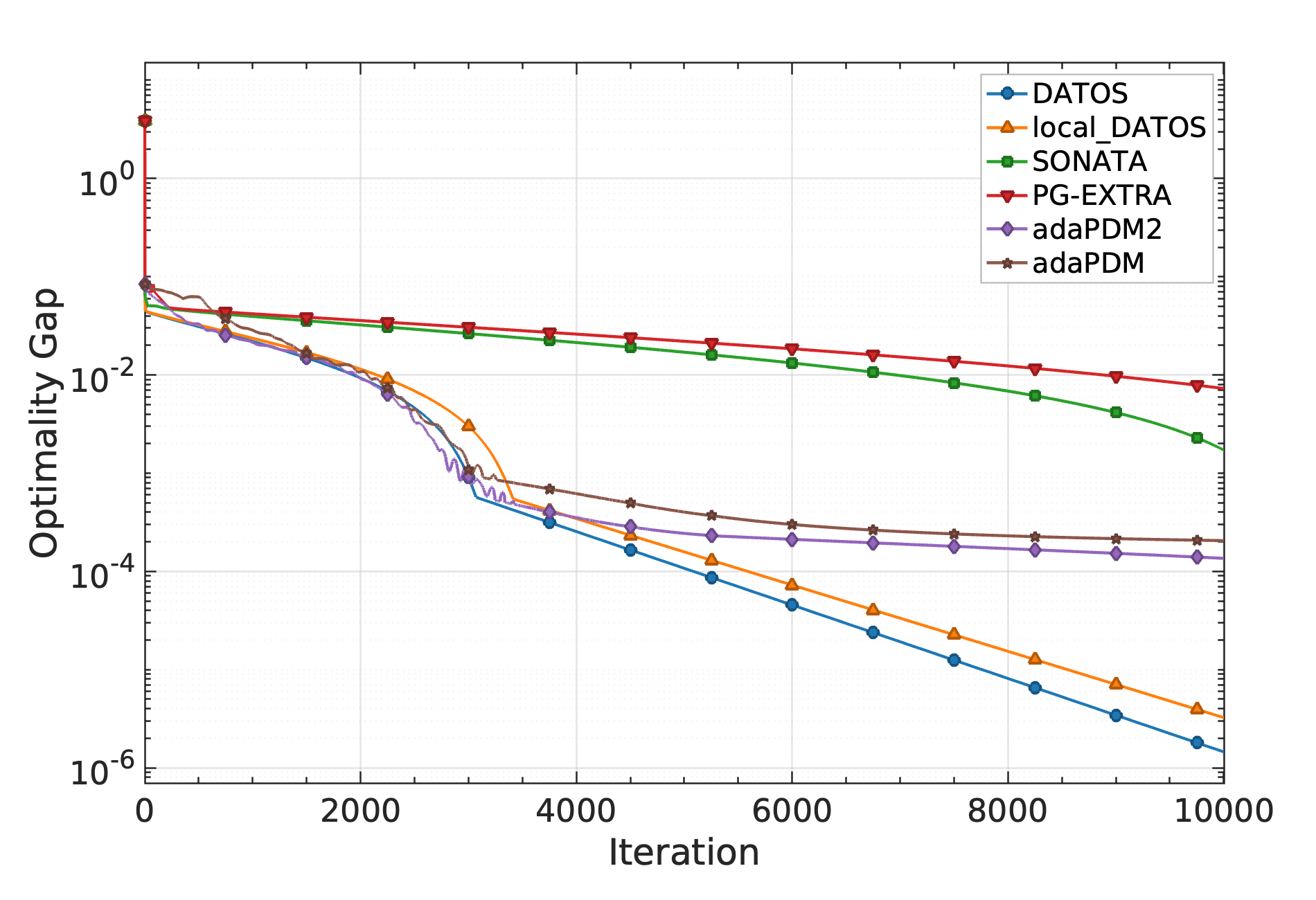}
    \caption{ML estimate of the covariance matrix: ${\frac{1}{m}\sum_{i=1}^m u(x_i)-u(x^*)}$ v.s. \# iterations.  Comparison of PG-EXTRA, SONATA, adaPDM, adaPDM2, global\_DATOS and local\_DATOS on Erdos-Renyi graphs with   edge-probability:  $p=0.1$  (left);   $p=0.5$ (middle); and  $p=0.9$ (right).}\vspace{-0.45cm}

\label{fig:general local smooth}
\end{figure*}

\subsection{Linear regression with elastic net regularization} We finally consider a strongly convex composite instance of~\eqref{eq:problem}, namely $\ell_2$-regularized least squares with an $\ell_1$ penalty:\vspace{-0.4cm}
\begin{equation}
    f_i(x)=\frac{1}{n}\|A_ix-b_i\|^2+\frac{\gamma_i}{2}\|x\|^2,\quad r(x)=\lambda\|x\|_1,\vspace{-0.1cm}
\end{equation}
where $(A_i,b_i)\in\mathbb{R}^{n\times d}\times\mathbb{R}^n$ is agent $i$'s local dataset. The entries of $A_i$ and $b_i$ are sampled i.i.d.\ from the standard normal distribution. We set $n=20$ (hence $N=400$), $d=500$, $\lambda=10^{-5}$, and $\gamma_i=0.1+(i-1)\times 0.1$, which yields heterogeneous local smoothness constants and a condition number of $f$ equal to  $\kappa\approx 82.62$.  \begin{figure*}[htbp]\vspace{-0.3cm}
    \centering
    \includegraphics[width=0.32\linewidth]{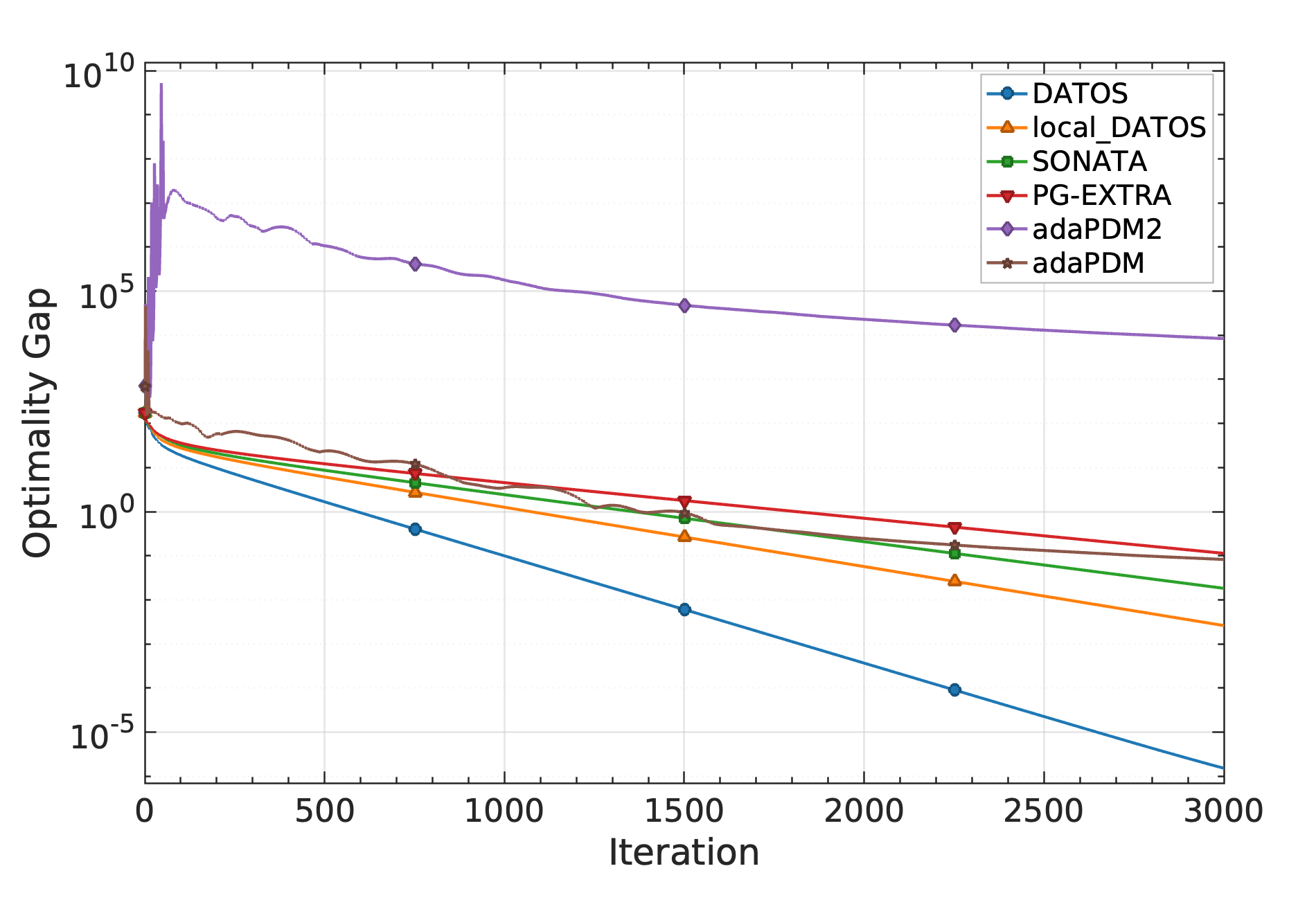}
\includegraphics[width=0.32\linewidth]{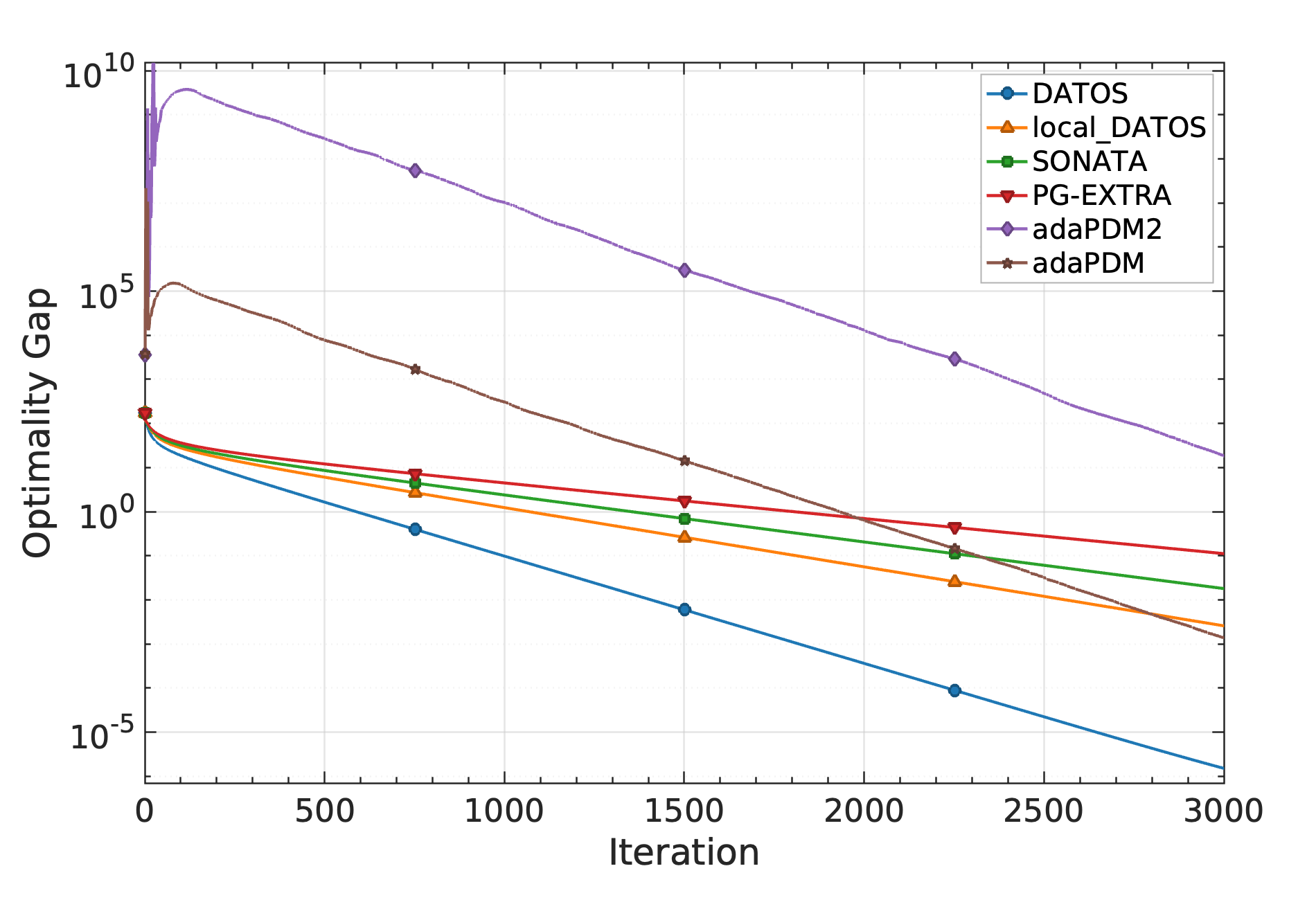}
\includegraphics[width=0.32\linewidth]{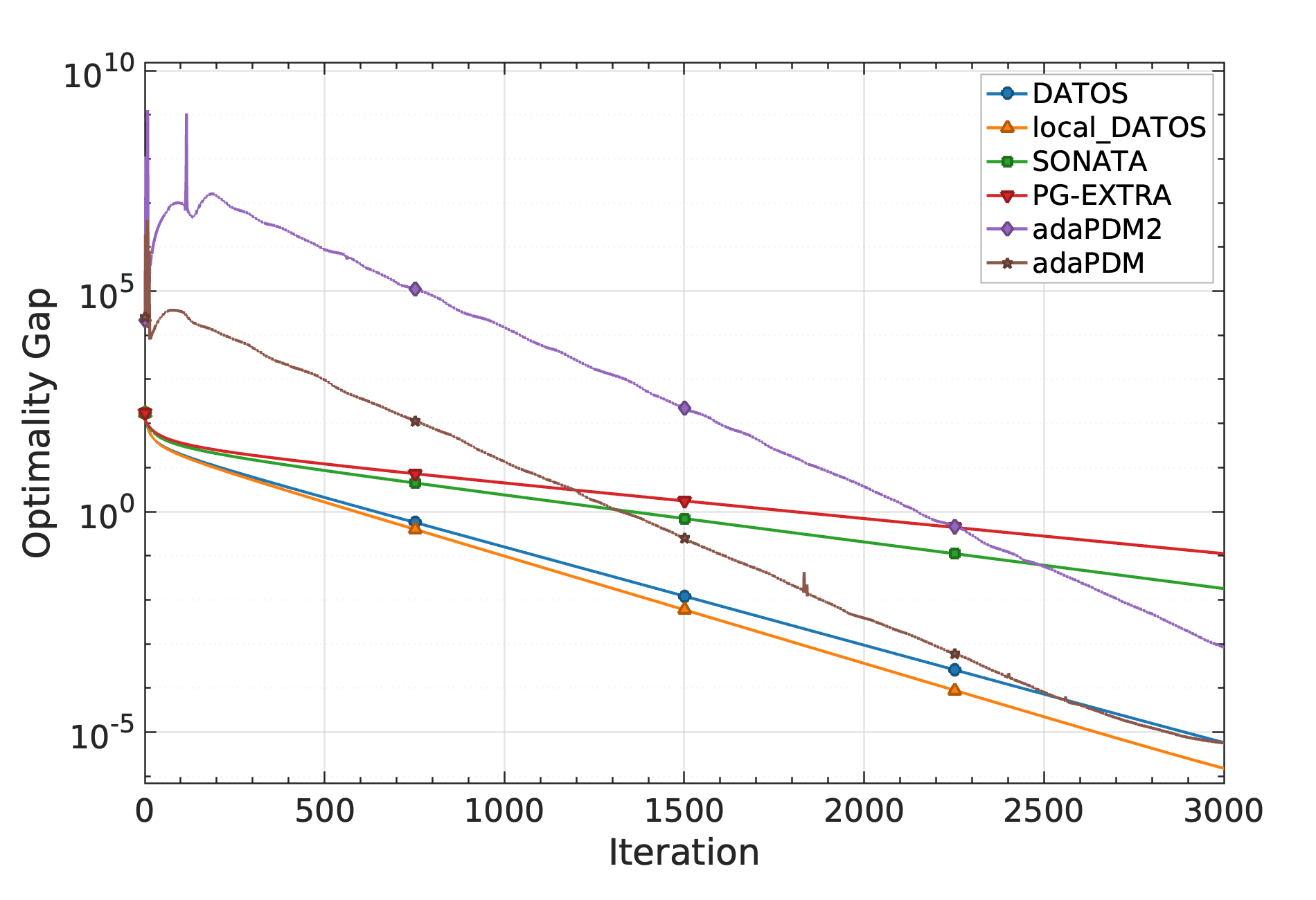}
\caption{Linear regression with elastic net regularization: ${\|\X^k-\X^*\|^2}$  v.s. \# iterations. Comparison of PG-EXTRA, SONATA, adaPDM, adaPDM2, global\_DATOS and local\_DATOS on Erdos-Renyi graphs with   edge-probability:  $p=0.1$  (left);   $p=0.5$ (middle); and  $p=0.9$ (right).}\vspace{-0.2cm} 
\label{fig:general scvx}
\end{figure*}

 Figure~\ref{fig:general scvx} reports the error   $\|\X^k-\X^\star\|^2$ versus the iteration counter. Consistent with the previous experiments, \texttt{global\_DATOS} and \texttt{local\_DATOS} improve substantially over the nonadaptive benchmarks, with the largest gains on sparse graphs. The observed slopes corroborate the linear convergence behavior predicted by our analysis.\vspace{-0.3cm}

\bibliographystyle{siamplain}
\bibliography{references}

@book{ryu2022large,
  title={Large-scale convex optimization: algorithms \& analyses via monotone operators},
  author={Ryu, Ernest K and Yin, Wotao},
  year={2022},
  publisher={Cambridge University Press}
}

@article{kuruzov2024achieving,
  title={Achieving linear convergence with parameter-free algorithms in decentralized optimization},
  author={Kuruzov, Ilya and Scutari, Gesualdo and Gasnikov, Alexander},
  journal={Advances in Neural Information Processing Systems},
  volume={37},
  pages={96011--96044},
  year={2024}
}

@article{polyak1969minimization,
  title={Minimization of unsmooth functionals},
  author={B.T. Polyak},
  journal={USSR Computational Mathematics and Mathematical Physics},
  volume={9},
  number={3},
  pages={14--29},
  year={1969},
  publisher={Elsevier},
  doi={10.1016/0041-5553(69)90085-5}
}

@article{BarzilaiBorwein1988,
  title={Two-point step size gradient methods},
  author={J. Barzilai and J. M. Borwein},
  journal={IMA Journal of Numerical Analysis},
  volume={8},
  number={1},
  pages={141--148},
  year={1988},
  publisher={Oxford University Press}
}

@inproceedings{Malitsky2019AdaptiveGD,
  title={Adaptive gradient descent without descent},
  author={Y. Malitsky and K. Mishchenko},
  booktitle={International Conference on Machine Learning},
  year={2019}
}

@article{li2024problem,
  title={Problem-Parameter-Free Decentralized Nonconvex Stochastic Optimization},
  author={J. Li  and X. Chen and S. Ma and M. Hong},
  journal={arXiv preprint arXiv:2402.08821},
  year={2024}
}

@article{Nazari_Tarzanagh_Michailidis_2022, title={DAdam: A Consensus-Based Distributed Adaptive Gradient Method for Online Optimization}, volume={70},   journal={IEEE Transactions on Signal Processing}, author={P. Nazari and D.A. Tarzanagh and G. Michailidis}, year={2022}, pages={6065–6079} }

@inproceedings{chen2023convergence,
  title={On the convergence of decentralized adaptive gradient methods},
  author={X. Chen and B. Karimi and W. Zhao and P. Li},
  booktitle={Asian Conference on Machine Learning},
  pages={217--232},
  year={2023},
  organization={PMLR}
}

@Article{Nedic_Olshevsky_Rabbat2018,
  author  = {A. Nedi\'{c} and  A. Olshevsky and M. Rabbat},
  title   = {Network topology and communication-computation tradeoffs in decentralized optimization},
  journal = {Proceedings of the IEEE},
  year    = {2018},
  volume  = {106},
  pages   = {953-976},
}

@article{Sayed-book,
  title={Adaptation, learning, and optimization over networks},
  author={A. H. Sayed},
  journal={Foundations and Trends in Machine Learning},
  volume={7},
  pages={311-801},
  year={2014},
  month={January},
  publisher={Now Publishers}
}

@article{shi2015proximal,
  title={A proximal gradient algorithm for decentralized composite optimization},
  author={Shi, W. and Ling, Q. and Wu, G. and Yin, W.},
  journal={IEEE Transactions on Signal Processing},
  volume={63},
  number={22},
  pages={6013--6023},
  year={2015},
  publisher={IEEE}
}

@inproceedings{Orabona19,
  title={On the convergence of stochastic gradient descent with adaptive stepsizes},
  author={Li, Xiaoyu and Orabona, Francesco},
  booktitle={The 22nd international conference on artificial intelligence and statistics},
  pages={983--992},
  year={2019},
  organization={PMLR}
}

@article{Ward20,
  title={AdaGrad stepsizes: Sharp convergence over nonconvex
landscapes},
  author={R. Ward and X. Wu and 
L. Bottou},
  journal={The Journal of Machine Learning Research},
  volume={21},
  pages={1--30},
  year={2020},
  publisher={JMLR. org}
}

@article{sun2019distributed,
	title={Distributed optimization based on gradient-tracking revisited: Enhancing convergence rate via surrogation},
	author={Sun, Y. and   Scutari, G.  and Daneshmand, A. },
	volume = {32},
	issue = {2},
	pages = {354--385},
	journal={SIAM J. on Optimization},
	year={2022}
}

@inproceedings{duchi2011adaptive,
  title={Adaptive subgradient methods for online learning and stochastic optimization},
  author={J. Duchi and E. Hazan and Y. Singer},
  booktitle={Proceedings of the 24th International Conference on Neural Information Processing Systems},
  pages={257--265},
  year={2011}
}

@inproceedings{Kim_Lim_Kim_2016, address={Prague, Czech Republic}, title={Low-Power, Long-Range, High-Data Transmission Using Wi-Fi and LoRa}, booktitle={2016 6th International Conference on IT Convergence and Security (ICITCS)},   author={D.H. Kim and J.Y. Lim and J.D. Kim}, year={2016}, month=sep, pages={1–3} }

@article{Xin_Pu_Nedic_Khan_2020, title={A General Framework for Decentralized Optimization With First-Order Methods}, volume={108},   number={11}, journal={Proceedings of the IEEE}, author={Xin, R. and Pu, S. and Nedic, A. and Khan, U. A.}, year={2020}, month=nov, pages={1869–1889} }

@article{Janssen_BniLam_Aernouts_Berkvens_Weyn_2020, title={LoRa 2.4 GHz Communication Link and Range}, volume={20}, rights={https://creativecommons.org/licenses/by/4.0/}, ISSN={1424-8220}, url={https://www.mdpi.com/1424-8220/20/16/4366}, DOI={10.3390/s20164366}, abstractNote={Recently, Semtech has released a Long Range (LoRa) chipset which operates at the globally available 2.4 GHz frequency band, on top of the existing sub-GHz, km-range offer, enabling hardware manufacturers to design region-independent chipsets. The SX1280 LoRa module promises an ultra-long communication range while withstanding heavy interference in this widely used band. In this paper, we first provide a mathematical description of the physical layer of LoRa in the 2.4 GHz band. Secondly, we investigate the maximum communication range of this technology in three different scenarios. Free space, indoor and urban path loss models are used to simulate the propagation of the 2.4 GHz LoRa modulated signal at different spreading factors and bandwidths. Additionally, we investigate the corresponding data rates. The results show a maximum range of 133 km in free space, 74 m in an indoor office-like environment and 443 m in an outdoor urban context. While a maximum data rate of 253.91 kbit/s can be achieved, the data rate at the longest possible range in every scenario equals 0.595 kbit/s. Due to the configurable bandwidth and lower data rates, LoRa outperforms other technologies in the 2.4 GHz band in terms of communication range. In addition, both communication and localization applications deployed in private LoRa networks can benefit from the increased bandwidth and localization accuracy of this system when compared to public sub-GHz networks.}, number={16}, journal={Sensors}, author={T. Janssen and N. BniLam and M. Aernouts and R. Berkvens and M. Weyn}, year={2020}, month=aug, pages={4366}, language={en} }

@inproceedings{Reddi2018OnTC,
  title={On the Convergence of Adam and Beyond},
  author={S. J. Reddi and S. Kale and S. Kumar},
   booktitle={International Conference on Learning Representations (ICLR)},
  year={2018}
}

@article{Malitsky_Mishchenko_2024, title={Adaptive Proximal Gradient Method for Convex Optimization},    journal={arXiv preprint arXiv:2308.02261},  
author={Y. Malitsky and K. Mishchenko}, year={2024}}

@article{Yura-Pock-LSPrimal-dual-18,
author = {Malitsky, Yura and Pock, Thomas},
title = {A First-Order Primal-Dual Algorithm with Linesearch},
journal = {SIAM Journal on Optimization},
volume = {28},
number = {1},
pages = {411-432},
year = {2018}  
}

@article{Zhou24, title={AdaBB: Adaptive Barzilai-Borwein method for convex optimization},    journal={arXiv preprint arXiv:2401.08024},  
author={D. Zhou  and  S. Ma  and J. Yang}, year={2024}}

@article{Latafat_23b, 
title={Adaptive proximal algorithms for convex optimization under local lipschitz continuity of the gradient}, 
journal={Mathematical Programming, Series A},
year={2024},
author={P. Latafat and A. Themelis and P. Patrinos}}

@article{kingma2014adam,
  title={Adam: A Method for Stochastic Optimization},
  author={D. P. Kingma and J. Ba},
  journal={CoRR},
  year={2014},
  volume={abs/1412.6980}
}

@article{Notarstefano24,
  title={Towards Parameter-free Distributed
Optimization: a Port-Hamiltonian Approach},
  author={R. Aldana-Lopez and A. Macchelli and G. Notarstefano and R. Aragues and  C. Sagues},
  journal={arXiv:2404.13529 },
  year={2024}
}

@article{davis2017three,
  title={A three-operator splitting scheme and its optimization applications},
  author={Davis, Damek and Yin, Wotao},
  journal={Set-valued and variational analysis},
  volume={25},
  pages={829--858},
  year={2017},
  publisher={Springer}
}

@book{bertsekas2016nonlinear,
  author    = {Dimitri P. Bertsekas},
  title     = {Nonlinear Programming},
  publisher = {Athena Scientific},
  edition   = {third},
  year      = {2016},
}

@article{o2020equivalence,
  title={On the equivalence of the primal-dual hybrid gradient method and Douglas--Rachford splitting},
  author={O’Connor, Daniel and Vandenberghe, Lieven},
  journal={Mathematical Programming},
  volume={179},
  number={1},
  pages={85--108},
  year={2020},
  publisher={Springer}
}

@article{chih2libsvm,
  title={" LIBSVM: a library for support vector machines," ACM Transactions on Intelligent Systems and Technology, 2: 27: 1--27: 27, 2011},
  author={Chih-Chung, CHANG},
  journal={http://www. csie. ntu. edu. tw/\~{} cjlin/libsvm},
  volume={2},
  year={1996},
}

@article{guo2023decentralized,
  title={Decentralized inexact proximal gradient method with network-independent stepsizes for convex composite optimization},
  author={Guo, Luyao and Shi, Xinli and Cao, Jinde and Wang, Zihao},
  journal={IEEE Transactions on Signal Processing},
  volume={71},
  pages={786--801},
  year={2023},
  publisher={IEEE}
}

@article{chambolle2011first,
  title={A first-order primal-dual algorithm for convex problems with applications to imaging},
  author={Chambolle, Antonin and Pock, Thomas},
  journal={Journal of mathematical imaging and vision},
  volume={40},
  pages={120--145},
  year={2011},
  publisher={Springer}
}

@article{xu2025accelerated,
  title={An Accelerated Primal Dual Algorithm with Backtracking for Decentralized Constrained Optimization},
  author={Xu, Qiushui and Aybat, Necdet Serhat and G{\"u}rb{\"u}zbalaban, Mert},
  journal={arXiv:2512.07085},
  year={2025}
}

@article{lewis2002active,
  title={Active sets, nonsmoothness, and sensitivity},
  author={Lewis, Adrian S},
  journal={SIAM Journal on Optimization},
  volume={13},
  number={3},
  pages={702--725},
  year={2002},
  publisher={SIAM}
}

@article{liang2017activity,
  title={Activity identification and local linear convergence of forward--backward-type methods},
  author={Liang, Jingwei and Fadili, Jalal and Peyr{\'e}, Gabriel},
  journal={SIAM Journal on Optimization},
  volume={27},
  number={1},
  pages={408--437},
  year={2017},
  publisher={SIAM}
}

@article{hare2004identifying,
  title={Identifying active constraints via partial smoothness and prox-regularity},
  author={Hare, Warren L and Lewis, Adrian S},
  journal={Journal of Convex Analysis},
  volume={11},
  number={2},
  pages={251--266},
  year={2004},
  publisher={HELDERMANN VERLAG LANGER GRABEN 17, 32657 LEMGO, GERMANY}
}

@book{rockafellar2009variational,
  title={Variational analysis},
  author={Rockafellar, R Tyrrell and Wets, Roger J-B},
  volume={317},
  year={2009},
  publisher={Springer Science \& Business Media}
}

@book{boumal2023introduction,
  title={An introduction to optimization on smooth manifolds},
  author={Boumal, Nicolas},
  year={2023},
  publisher={Cambridge University Press}
}

@article{chen2025parameter,
  title={A Parameter-free Decentralized Algorithm for Composite Convex Optimization},
  author={Chen, Xiaokai and Kuruzov, Ilya and Scutari, Gesualdo and Gasnikov, Alexander},
  journal={arXiv preprint arXiv:2508.01466},
  year={2025}
}

@article{jordan6muon,
  title={Muon: An optimizer for hidden layers in neural networks, 2024},
  author={Jordan, Keller and Jin, Yuchen and Boza, Vlado and Jiacheng, You and Cecista, Franz and Newhouse, Laker and Bernstein, Jeremy},
  journal={URL https://kellerjordan. github. io/posts/muon},
  volume={6},
  year={2024}
}
\end{document}